\documentclass[reqno,a4paper, 11pt]{amsart}
\usepackage[english]{babel}
\usepackage[utf8]{inputenc}

\usepackage[a4paper=true,pdfpagelabels]{hyperref}
\usepackage{graphicx}

\usepackage{epsfig}
\usepackage{latexsym}
\usepackage{mathabx}
\usepackage{amsthm,amsfonts,amsmath,mathrsfs,amssymb}

\usepackage{color}
\def\a{\alpha}       \def\b{\beta}        
            \def\om{\omega}

\def\D{{\mathbb D}}  \def\T{{\mathbb T}}
\def\C{{\mathbb C}}  \def\N{{\mathbb N}}

\def\({\left(}       \def\){\right)}

\renewcommand{\H}{\mathcal{H}}
\newcommand{\I}{\mathcal{I}}

\newtheorem{theorem}{Theorem}[section]
\newtheorem{lemma}[theorem]{Lemma}
\newtheorem{proposition}[theorem]{Proposition}
\newtheorem{corollary}[theorem]{Corollary}
\newtheorem{lettertheorem}{Theorem}
\newtheorem{letterlemma}[lettertheorem]{Lemma}

\theoremstyle{definition}
\newtheorem{definition}[theorem]{Definition}
\newtheorem{example}[theorem]{Example}

\theoremstyle{remark}
\newtheorem{remark}[theorem]{Remark}

\numberwithin{equation}{section}

\newenvironment{Prf}{\noindent{\emph{Proof of}}}
{\hfill$\Box$ }

\DeclareMathOperator*{\esssup}{ess\,sup}

\begin{document}

\title[The Hilbert matrix on analytic tent spaces]{The Hilbert matrix on analytic tent spaces}

\author[Tanausú Aguilar-Hernández]{Tanausú Aguilar-Hernández}
\address{Departamento de Matemática Aplicada, Universidad de Málaga, Campus de Teatinos, 29071, Málaga, Spain}
\email{taguilar@uma.es}
\author[Petros Galanopoulos]{Petros Galanopoulos}
\address{Department of Mathematics, Aristotle University of Thessaloniki, 54124, Thessaloniki, Greece}
\email{petrosgala@math.auth.gr}
\author[Elena de la Rosa]{Elena de la Rosa}
\address{Departamento de Matemáticas Fundamentales, Facultad de Ciencias, Universidad Nacional de Educación a Distancia, Madrid, Spain}
\email{elenadelarosa@mat.uned.es}



\thanks{The first author was partially supported by Ministerio de Innovación y Ciencia, Spain, project PID2022-136320NB-I00. The second author  was partially supported by the Hellenic Foundation for Research and Innovation (H.F.R.I.) under the 2nd Call for Research Projects to support Faculty Members and Researchers (Project Number 4662). The third author was partially supported by Ministerio de Ciencia e Innovacion, Spain, project PID2022-136619NB-I00; and La Junta de Andalucía, project FQM210. }

\maketitle

\begin{abstract}
 We study for the first time the action of the Hilbert matrix $$\H=(c_{n,k})_{n,k\geq 0}, \quad c_{n,k}=\frac{1}{n+k+1}$$ 
 on the analytic tent spaces $AT^q_p, 1<p,q <\infty,$ of the unit disc $\mathbb D$ of the complex plane. They were proposed by Triebel as the natural analytic version of the tent spaces of measurable functions defined by Coifman, Meyer and Stein. The $AT_{p}^{q}$ spaces are consisted of those analytic functions $f$ in $\mathbb D$ such that  
 \begin{align*}
 \|f\|_{AT_{p}^{q}}= \left\{\int_{\T} \left(\int_{\Gamma_{1/2}(\xi)} |f(z)|^p \ \frac{dA(z)}{1-|z|^2} \right)^{q/p}\ |d\xi|\right \}^{1/q}<+\infty,
  \end{align*}
 where
 \begin{align*}
 	\Gamma_{1/2}(\xi) =\bigl\{ z\in \mathbb{D} : |z|< 1/2   \bigr\} \cup \bigcup_{|z|<1/2}[z,\xi),
 \end{align*}
$dA(z)$ is the normalized area Lebesgue measure in $\mathbb D$ and $|d\xi|$ is the arc length in the unit circle $\mathbb T$. The Bergman spaces $A^p, p>1,$ stand among the $AT_{p}^{q}$ and correspond to the case $p=q$. The multiplication of the Hilbert matrix with the column matrix with entries the Taylor coefficients of an $f(z)=\sum_{k\geq 0} a_k z^k $ analytic in $\mathbb D$ introduces the series
\begin{equation*}
	\H(f)(z)=
	\sum_{n=0}^{\infty}\left(\sum_{k=0}^{\infty}
	\frac{a_k}{n+k+1}\right)z^n\,, \quad z\in \mathbb D\,\,
\end{equation*}
known in the literature as Hilbert operator. We prove that it is a bounded operator
on the $AT_{p}^{q}$ when $1/p + 1/q <1,\, p>2$. This is a natural range for the values of the indices $p,q$ compared to what is known in the special case of the Bergman spaces.
 We confront the question 
under discussion through a more general point of view by studying an associated integral operator defined with respect to a positive Borel measure $\mu$ on $[0,1)$. Finally, we provide an estimation of the norm of the Hilbert operator. Our work extends in a non-trivially way previous results on the Bergman spaces to the analytic tent spaces. 
\end{abstract}

\section{Introduction}

The Hilbert matrix is the infinite Hankel matrix 
$$\H=\left(%
\begin{array}{ccccc}
            1 & \frac{1}{2}  & \frac{1}{3}   &. \\
  \frac{1}{2} & \frac{1}{3}  & \frac{1}{4}  &.  \\
  \frac{1}{3} & \frac{1}{4}  & \frac{1}{5}  & . \\
  .  & . & . & .  \\
\end{array}
\right).
$$
The  multiplication of $\H$ with the column matrix formed by the terms of a complex sequence $\{a_k\}_k$ defines the operator
\begin{equation}\label{hilbert lp}
	\{a_k\}_k\,\, \mapsto \Big\{\sum_{k=0}^{\infty}
	\frac{a_k}{n+k+1}\Big\}_n, \quad n\in\N\cup\{0\}\,.
\end{equation}
Hilbert's inequality implies that the operator (\ref{hilbert lp}) is bounded on the $\ell^p$ spaces when $p \in (1,\infty),$ (see \cite{HLP}). Taking into account that $\ell^2$ is identified to the Hardy space $H^2$ we can equivalently state that (\ref{hilbert lp}) introduces a bounded operator on $H^2$ in the following sense. 
Let $\mathbb D$ be the unit disc in the complex plane and $ \mathcal H(\mathbb D)$ be the  analytic functions in $\mathbb D.$ We say that an $f(z)=\sum_{k=0}^\infty a_kz^k\in \mathcal H(\D) $ belongs to $H^2$ if and only if 
\begin{equation}\label{Norm H2}
\|f\|^2_{H^2} = \sum_{k\geq 0}\, |a_k|^2 < \infty. 
\end{equation}
Considering (\ref{hilbert lp}) on the Taylor coefficients of an $f \in H^2$
we can formally define the series
\begin{equation}\label{H}
\H(f)(z)=
\sum_{n=0}^{\infty}\left(\sum_{k=0}^{\infty}
\frac{a_k}{n+k+1}\right)z^n\,, \quad z\in \mathbb D\,.
\end{equation}
Just an application of the Cauchy-Schwartz inequality on the inner sum
is sufficient to verify that the series (\ref{H})
is well defined and represents an analytic function in $\mathbb D$. 
As a result it is defined the operator 
\begin{equation}\label{hilbert H2}
	f\,\, \mapsto \H(f)
\end{equation}
which is bounded in $H^2$.

The linear operator (\ref{hilbert H2}) is known in the literature as Hilbert operator. Its study has been extended to the Hardy spaces $H^p, p\in [1,\infty)$, which are  consisted of those $f\in \mathcal H(\mathbb D)$ that 
\begin{equation}\label{Norm Hp}
\|f\|_{H^p} = \sup_{r\in(0,1)}  \Big\{\frac {1}{2\pi} \int_0^{2\pi}\, |f(re^{i\theta})|^p\,d\theta \Big\}^{1/p}\,<\infty\,.
\end{equation}
When $p=2$ the two norms (\ref{Norm H2}) and (\ref{Norm Hp}) are identified.
If $p\neq 2$ the well definition of the series (\ref{H}) and the fact that it represents an analytic function in $\mathbb D$ come from the strength of Hardy's inequality
\begin{equation}\label{hardy's inequality}
	\sum_{k\geq 0} \frac{|a_k|}{k+1} \leq \pi \,\|f\|_{H^1}
\end{equation}
and the inclusions
$
H^q \subset H^p \subset H^1 \,, \,\,\,1<p<q\,.
$ 
 For more information on the Hardy spaces theory, we suggest \cite{Du}.
 
However, Hilbert's inequality no longer serves
in order to prove the boundedness of (\ref{hilbert H2}) when $ p\neq 2$. The key for the study of the Hilbert operator on these spaces 
is the identity 
\begin{equation}\label{identity}
	\H(f)(z)=\sum_{n=0}^{\infty}\left(\sum_{k=0}^{\infty}
	\frac{a_k}{n+k+1}\right)z^n=\int_0^1 \frac{f(t)}{1-tz} \,dt, \quad z\in \mathbb D
	\end{equation}
which is easy to check that it is true when $f$ is an analytic polynomial, but to say that it holds for any $f\in H^p, p\in [1,\infty)$, we have to apply once more Hardy's inequality. 
As a consequence of (\ref{identity}),  in order to study the Hilbert operator on the Hardy spaces we employ
 the integral operator  
 \begin{equation}\label{I}
 \I(f)(z) = \int_0^1 \frac{f(t)}{1-tz}\, dt, \quad z \in \mathbb D .
\end{equation}
 Being that the case, it turns out that the Hilbert operator is bounded on the $H^p, p>1.$ Boundedness brakes down in $H^1$. 
 See \cite{DiS, DJV}.
 
 The operator (\ref{I}) is the key for the study of the Hilbert operator on the Bergman spaces $A^p, p\in (2,\infty)$, as well. 
 We say that an $f\in \mathcal H(\mathbb D)$ belongs to the Bergman space $A^p$ if 
 $$
 \|f\|_{A^p} =  \Big\{\int_{\mathbb D}|f(z)|^p\,dA(z)  \Big\}^{1/p}<\infty,
 $$
 that is, if $f$ is $p-$integrable  with respect to the normalized area Lebesgue measure on $\mathbb D$. 
 Notice that if $f(z)=\sum_{k\geq 0} a_k z^k \in A^p, p>2$, then 
 \begin{equation}\label{hardy's inequality Bergman}
 \sum_{k\geq 0} \frac{|a_k|}{k+1} \leq C \|f\|_{A^p}\,,
 \end{equation}
(see \cite[p. 307]{ArJeVu}, \cite[Lemma 4.1]{NowPav}) 
On account of condition (\ref{hardy's inequality Bergman}), we can verify that 
the series (\ref{H}) are well defined and that (\ref{identity}) holds for any  $f\in A^p$, $p>2$. Consequently, we are allowed to engage appropriately the integral operator (\ref{I})  and get that the Hilbert operator is bounded on any $A^p, p>2$ (see \cite{Di}). In the case  $p=2$, the series (\ref{H}) is not even well defined. Indeed, if $f(z)=\sum_{k\geq 0} a_k z^k$ then 
$$
\|f\|_{A^2}^2= \sum_{k\geq 0} \frac{1}{k+1} \,|a_k|^2 \,.
$$
Choosing the 
 $$
f(z)=\sum_{k\geq 0} \frac{1}{\log(k+1)} z^k \in A^2 \,,
$$
we get that for any $n\in \N$
$$
\sum_{k\geq 0} \frac{1}{\log(k+1) (n+k+1)} =\infty
$$
(see \cite{Di}, \cite[p. 2810]{DJV}).

Our aim is to extend the study of the Hilbert operator to the analytic tent spaces $AT^q_p$, when $1<p,q<\infty$. 
The Bergman spaces  stand among them and the Hardy spaces are identified to a limit case. Let a $\xi\in \T$. We define the cone-like region
\begin{align*}
	\Gamma_{1/2}(\xi) =\bigl\{ z\in \mathbb{D} : |z|< 1/2   \bigr\} \cup \bigcup_{|z|<1/2}[z,\xi)\,. 
\end{align*}
The  $AT_p^q $ spaces consist of those $f \in \mathcal{H}(\D)$ such that 
\begin{align}\label{TentA}
	\|f\|_{AT_{p}^{q}}= \left\{\int_{\T} \left(\int_{\Gamma_{1/2}(\xi)} |f(z)|^p \ \frac{dA(z)}{(1-|z|^2)} \right)^{q/p}\ |d\xi|\right \}^{1/q}<+\infty.
\end{align}
If $p=q$ then $AT^p_p \equiv A^p$ as it is verified by an application of Fubini's Theorem.
The limit case $AT^q_{\infty}$ defined as 
\begin{align}\label{TentA}
	\|f\|_{AT_{\infty}^{q}}= \left\{\int_{\T} 
	\left(\sup_{z \in \Gamma_{1/2}(\xi)} |f(z)|\right)^q 
	|d\xi|\right \}^{1/q}<+\infty
\end{align}
are exactly the Hardy spaces  $H^q$, $1< q < \infty$ (see, e.g., \cite{Rudin,Pav}).
The work of Coifman, Meyer and Stein is considered as the starting point for the theory of the tent spaces $T^q_p$ of measurable functions (see \cite{CMS}). Triebel introduced the $AT^q_p$ as the 
natural  analytic version of them. We provide more information on the subject in Section $2$.
Since then, the analytic tent spaces have been widely studied by many authors 
(see, e.g.,  \cite{Arsenovic, CohnVerbitsky, Jevtic_1996, Luecking1987, Luecking, MiihkinenPauPeralaWang2020, Perala} ).


In the frames of this work, we focus  on the study of the action of the Hilbert operator on $AT^q_p$ when $ 1<p,q < \infty, 1/p+1/q <1$ since otherwise it  is not even  well defined. For instance, consider the function
$$
f(z)=\frac{1}{1-z}=\sum_{k\geq 0} z^k \,.
$$ 
It is proved in \cite[Example 2.3]{AgConRodr} that $f\in AT^q_p,\, 1/p+1/q >1$. So, for that $f$ the coefficients of the series (\ref{H}) are not defined since 
$$
\sum_{k\geq 0} \frac{1}{n+k+1}=\infty,\quad n\in \mathbb N\,.
$$
Even in the limit case  $1/p+1/q =1$ it appears a similar irregularity. Since $A^2\subseteq AT^q_p$  when $ p\in (1,2]$ (see \cite[Theorem 3.3]{AgConRodr}), it suffices to consider the
$$
f(z)=\sum_{k\geq 0} \frac{1}{\log(n+1)} z^n \in A^2 \,.
$$
That is enough in order to realize that the Hilbert operator is not defined in the case $1/p+1/q =1$ and $p\in (1,2] $. 

Actually, the same test function serves for any $AT^q_p, 1/p+1/q=1$ as we make clear below. According to Theorem $2.31$ in \cite{Zygmund}, we are aware that  if
$$
f(z)=\frac{1}{(1-z)\log\left(\frac{e}{1-z}\right)}=\sum_{n\geq 0} A_n z^n ,\,\,z\in \D
$$
then 
$$
A_n \simeq \frac{1}{\log(n+1)}\,, \, \,\,n\to \infty\,.
$$
 However, it is needed more elaboration to prove now that $f \in AT^q_p, \,1/p+1/q =1, p \in (2,\infty)$  which in its turn establishes that the series (\ref{H}) is not defined in any of them too. In Section $3$, we prove the following:
 \begin{proposition}\label{test function}
 	Let $1< p,q<\infty$, $\frac{1}{p}+\frac{1}{q}= 1$ and $p>2$ then 
 	$$ f(z)=\frac{1}{(1-z)\log\left(\frac{e}{1-z}\right)}\,\in \,AT^q_p \,.$$
 	 \end{proposition}


On the other hand, one can check easily by means of the following result that the Hilbert operator is well defined when $1/p + 1/q <1$.
  
 \begin{proposition}\label{first term}
 	Let $1/p + 1/q <1$ and $n\in\N$. If $f(z)=\sum_{k\geq 0}a_k z^k \in AT^q_p$ then  
 		the series 
 		$$
 		\sum_{k\geq 0} \frac{a_k}{n+k+1}
 		$$
 		converges. Moreover, there is a positive constant $C=C(p,q)$ such that
 		$$
 		\sup_{n\in\N} \left| \sum_{k\geq 0} \frac{a_k}{n+k+1}\right|\leq C \|f\|_{AT_{p}^{q}}.
 		$$
 \end{proposition} 
  
The next step is to deal with the boundedness of the Hilbert operator. Our point of view is more general in the following sense. We consider the broad class of   Hankel matrices 
 \begin{equation}\label{Hm}
 	\H_{\mu}=(\mu_{n,k})_{n,k\geq 0},\,\,\,\,\,\,\text{where}\,\,\,\,\,\mu_{n,k}=\mu_{n+k},
 \end{equation}
introduced by the moments $\mu_n=\int_0^1 t^n d\mu(t), \, n=0,1,2, \dots$ of a finite positive Borel measure  $\mu$ on $[0,1)$. A matrix like that is known as generalized Hilbert matrix. It is evident that $\H_{\mu}\equiv \H$ when $\mu$ is  the Lebesgue measure on $[0,1)$. 

 Widom, in connection to the  Hamburger moment problem, 
   proved that  (\ref{Hm}) acts boundedly on $\ell^2$ if and only if the measure $\mu$ is a 1-Carleson Measure ($1-$CM) on $[0,1)$. We recall  that a positive Borel measure $\mu$ on $\D$ is an $1-$CM if 
  \begin{equation}\label{1CM on D}
  	\sup_{I\subseteq \partial \D} \frac{\mu(S(I))}{|I|} <\infty\,,
  \end{equation}
  where $S(I)=\{ z=|z|e^{i\theta}\in \D : 1-|I| <|z| <1, \theta \in I\}$ and the $|I|$ stands for the arc length of the arc $I \subseteq \partial D$. 
  The little-$o$ version of (\ref{1CM on D}) characterizes the measure as 1-Vanishing Carleson measure (1-VCM).
  In the case of a measure $\mu$ supported on $[0,1)$, condition (\ref{1CM on D}) is stated as 
  \begin{equation}\label{1CM}
  	\sup_{t \in [0,1)} \frac{\mu([t,1))}{(1-t)} <\infty\,,
  \end{equation}
  or equivalently as 
  \begin{equation}\label{growth of the moments}
  	\mu_n = O(1/n) <\infty ,
  \end{equation}
(see \cite{Widom}).
In \cite{Power}, Power extended the results of Widom to the case of a complex measure $\mu $ on $\D$.
  
  Thinking the same way as before only now using the $\H_{\mu}$ of a positive Borel measure $\mu$ on $[0,1)$,
   we can define on the analytic polynomials $f(z)=\sum_{k\geq 0} a_k z^k $ the  operator 
\begin{equation}\label{Hm(f)}
	\H_{\mu}(f)(z)=\sum_{n=0}^{\infty}\left(\sum_{k=0}^{\infty}
	\mu_{n+k} a_k \right)z^n\,, \quad z\in \mathbb D\,.
\end{equation}
 We will refer to (\ref{Hm(f)}) as the generalized Hilbert operator. 
 On top of that, for such an $f$, the series represents an analytic function in $\D$ and can be identified to a well defined integral operator as 
 \begin{equation}\label{identity Hm}
 	\H_{\mu}(f)(z)=\sum_{n=0}^{\infty}\left(\sum_{k=0}^{\infty}
 	\mu_{n+k} a_k \right)z^n=\int_0^1 \frac{f(t)}{1-tz} \,d\mu(t)=:\mathcal{I}_\mu, \quad z\in \mathbb D\,.
 \end{equation}
In addition, a Carleson measure assumption on the measure provides the following:
 \begin{proposition}\label{Hmu=Imu}
	Let an $f\in AT^q_p$, $\frac{1}{p}+\frac{1}{q}<1$ and $p>2.$
	If $\mu$ is an  $1$-CM on $[0,1)$ then
	\begin{itemize}
		\item[(i)] for each $n\in \N$
		$$
		\sum_{k\geq 0} \mu_{n+k} |a_k| < \infty
		$$
		\item[(ii)]
		the power series 
		\begin{equation*}
			\mathcal H_{\mu}(f)(z)=\sum_{n=0}^{\infty}\left(\sum_{k=0}^{\infty}
			\mu_{n+k} a_k \right)z^n, \quad z\in \mathbb D\,
		\end{equation*}
		represents an analytic function in the unit disc. 
		\item[(iii)]	
		\begin{equation*}
			\H_{\mu}(f)(z)= \mathcal{I}_{\mu}(f)(z)=\int_0^1\frac{f(t)}{1-tz}d\mu(t),\, z\in \D\,.
		\end{equation*}		
	\end{itemize}
\end{proposition}
Note that, for the concrete case of the Lebesgue measure on $[0,1)$, case $(i)$ above implies that if $f(z)=\sum_{k\geq 0} a_k z^k \in AT^q_p$, $\frac{1}{p}+\frac{1}{q}<1$ and $p>2$, then  
 $$\sum_{k\geq 0} \frac{|a_k|}{k+1}$$  
 converges. Nevertheless, when $\frac{1}{p}+\frac{1}{q}<1$ we can verify the convergence of the series $\sum_{k\geq 0} \frac{a_k}{k+1}$ (see Proposition~\ref{first term}). It is worth noting that, although we know when the series  converges, it is not clear to us if the absolute convergence of the series  remains true when $1<p \leq 2$.\newline

The integral operator
\begin{equation}\label{Im(f)}
	\I_{\mu}(f)(z)=\int_0^1\frac{f(t)}{1-tz}d\mu(t), \quad z\in \mathbb D,
\end{equation}
results to be the key tool of our work.
 Several researchers have contributed to the study of $\H_{\mu}$ and $\I_\mu$  on classical function spaces \cite{Ch-Gi-Pe,Ga-Pe2010,GM1,GM2,GM3}. 
 
 Our first concern is the well definition of $\mathcal I_{\mu}$ independently of $\mathcal H_{\mu}$. In Section 3, we prove the following:
 \begin{proposition}\label{well definition of Imu}
 	Let $1<p,q <\infty$ and ${1/p+1/q <1}$ then 
 	$\mathcal I_{\mu}$ is well defined  on $AT^q_p$ if and  only if 
 	 \begin{equation}\label{final}
 			\sum_{n\geq 0}\, \left[\mu\left(\left[1-\frac{1}{2^n}, 1-\frac{1}{2^{n+1}}\right)\right) 2^{n(1/p+1/q)}\right]^{q'} \,<\,\infty ,
 		 \end{equation}
 	 where $1/q+1/q' =1$.
 \end{proposition}
 
 Bringing in mind the study of $\mathcal I_{\mu}$ on
 other spaces of analytic functions in the unit disc, one could come to the conclusion  that the measures carrying property (\ref{final}) are exactly those that imply boundedness. However, this not the case for the analytic tent spaces, not  all measures that satisfy (\ref{final}) introduce $\mathcal I_{\mu}$ 
 as bounded operator. We are able to prove the following theorem.
 \begin{theorem}
 \label{Theorem I}
 	Let $1<p,q<\infty$ with $\frac{1}{p}+\frac{1}{q}<1$ and let $\mu$ be a finite positive Borel measure on $[0,1)$. The following assertions hold:
 	\begin{itemize}
 		\item[(i)] The operator  $\mathcal{I}_{\mu}$  is bounded on $AT^q_p$ if and only if $\mu$ is an $1$-CM.
 		\item[(ii)] The operator  $\mathcal{I}_{\mu}$ is compact on $AT^q_p$ if and only if $\mu$ is an $1$-VCM. 
 	\end{itemize}
 \end{theorem}
 Notice that being a measure $1-CM$ on $[0,1)$ then condition (\ref{final}) is fulfilled. Conversely, there are measures that satisfy (\ref{final}) yet are not $1-CM$ on $[0,1)$. A typical example is the $d\mu(t)= \log \frac{1}{1-t}dt$.
 
 In relation to the generalized Hilbert operator the study of $\mathcal I_{\mu}$ leads to the following.

  \begin{theorem}\label{about Hm}
 	Let $1<p,q<\infty$ such that $\frac{1}{p}+\frac{1}{q}<1$ and $p>2$.
 	\begin{itemize}
 		\item[(i)] The operator $\mathcal{H}_{\mu}$ is well defined and  bounded on $AT^q_p$ if and only if $\mu$ is an $1$-CM.
 		\item[(ii)] The operator  $\mathcal{H}_{\mu}$ is compact on $AT^q_p$ if and only if $\mu$ is an $1$-VCM. 
 	\end{itemize}
 \end{theorem}
In view of Proposition~\ref{Hmu=Imu} and Theorem~\ref{Theorem I},  the sufficiency of the Carleson measure conditions in both cases above comes out immediately. We establish the necessity in Section 3. Moreover, reading the proofs it is clear that these conditions are necessary even if $p\in (1,2]$.

All the above come down to the following conclusion.
 \begin{corollary}
 	Let $1<p,q<\infty$ such that $\frac{1}{p}+\frac{1}{q}<1$ and $p>2$. The Hilbert operator $\mathcal{H}: AT^q_p\rightarrow AT^q_p$ is bounded but not compact.
 \end{corollary}

Finally, we provide an estimation of the norm of the Hilbert operator as stated below. The proof of the upper bound makes use of several powerful tools, among them the factorization of tent spaces  \cite{CohnVerbitsky}.  In addition, as a step toward the main result, we obtain an estimate for the norm of composition operators on certain tent spaces (see Lemma~\ref{composition on triebel spaces}). Regarding to the lower bound, we will follow the same reasoning as \cite{DJV}.
\begin{theorem}
\label{norm estimation}
Let $1<p,q<\infty$ such that $p\neq q$, $\frac{1}{p}+\frac{1}{q}<1$ and $p >  2$. It is true that
		$$\int_0^1 \frac{dt}{t^{1-\left(\frac{1}{p}+\frac{1}{q}\right)}(1-t)^{\frac{1}{p}+\frac{1}{q}}}\lesssim \|\mathcal{H}\|_{AT^q_p \to AT^q_p}\lesssim \int_0^1 \frac{dt}{t^{1-\frac{1}{q}}(1-t)^{\frac{1}{p}+\frac{1}{q}}} .$$
\end{theorem}
At this point we have to comment that the lower estimate generalizes known results related to the norm of the Hilbert operator on the Hardy and weighted Bergman spaces (see \cite{K,DJV,LMN}). We conjecture that this is the desired estimate from above as well.

The structure of the paper is as follows. 
Section 2 is devoted to the presentation of all the necessary information about the tent spaces theory. In Section 3 we present all the proofs of the results stated above. In Section 4 we expose the proof for the norm estimate of the classical Hilbert operator.

In what follows,  $A \lesssim B$ stands for the inequality $A \leq C B$ for some positive constant $C>0$ and $ A\gtrsim B $ analogously. If both $A \lesssim B$ and $A \gtrsim B$ hold, we write $A \asymp B$. This notation has already been used in the introduction.

\section{Previous definitions and results}
In this section, we recall the definition of the spaces of analytic functions under discussion. Moreover, we compile some properties for the sake of being self-contained.

\subsection{Tent spaces}
The work of Coifman, Meyer and Stein  is considered as the starting point of the study of tent spaces \cite{CMS}. 
Since then, they have been widely studied by many authors (see, e.g., \cite{Arsenovic, CohnVerbitsky, Jevtic_1996, Luecking1987, Luecking, MiihkinenPauPeralaWang2020, Perala}).

Let a $\xi\in \T$. We define the cone-like region
\begin{align*}
\Gamma_{1/2}(\xi) =\bigl\{ z\in \mathbb{D} : |z|< 1/2   \bigr\} \cup \bigcup_{|z|<1/2}[z,\xi)\,. 
\end{align*}
\begin{definition}
	Let $0<p,q<+\infty$. The tent spaces  $T_p^q$ consist of measurable functions $f$ on $\D$ such that 
	\begin{align}\label{TentA}
	\|f\|_{T_{p}^{q}}= \left\{\int_{\T} \left(\int_{\Gamma_{1/2}(\xi)} |f(z)|^p \ \frac{dA(z)}{(1-|z|^2)^{}} \right)^{q/p}\ |d\xi|\right \}^{1/q}<+\infty.
	\end{align}
 \end{definition}

It is of great importance that in \eqref{TentA} we can  use any of the cone-like regions
\begin{align*}
	\Gamma_{\beta}(\xi) =\left\{ z\in \mathbb{D} : |z|< \beta   \right\} \cup \bigcup_{|z|<\beta}[z,\zeta)\,,\quad \beta\in(0,1).
\end{align*}
In other words, we have 
\begin{align*}
	\|f\|_{T^q_p} \asymp \left\{\int_{\T} \left(\int_{\Gamma_{\beta}(\xi)} |f(z)|^p \ \frac{dA(z)}{(1-|z|^2)^{}} \right)^{q/p}\ |d\xi|\right \}^{1/q} \,.
\end{align*}
Actually, we can use any of the non-tangential regions 
 \begin{align*}
		 \Gamma_M (\xi) =\left\{z\in\D: |z-\xi|< M (1-|z|^2)  \right\},\quad  
          M>\frac{1}{2}\,\,
\end{align*}
as well.
This is true due the following technical lemma, which is  well-known to the experts of the area. 
The symbol $\Gamma_C(\xi)$ stands for any of the $\Gamma_{\beta} (\xi)$ or $\Gamma_M (\xi)$\,.
\begin{letterlemma}{\cite[Lemma 4, p. 66]{Arsenovic}}\label{estimate 1} 
	Let $0<p,q<+\infty$, $\lambda>\max\{1,p/q\}$ and $\mu$ be a positive Borel measure on $\D$\,. There are constants $K_i=K_i(p,q,\lambda,C),\,i=1,2 $ such that
	\begin{align*}
K_1 \int_{\T} \mu (\Gamma_C(\xi))^{q/p} |d\xi|\leq\int_{\T}\left(\int_{\D} \left(\frac{1-|z|}{|1-z\overline{\xi}|}\right)^{\lambda} \ d\mu(z)\right)^{q/p}|d\xi|\leq K_2 \int_{\T} \mu (\Gamma_C(\xi))^{q/p} |d\xi|\,.
	\end{align*}
\end{letterlemma}

 	The tent spaces are defined also for the limit values of the $p,q$ as well. The $T_{\infty}^{q}$ consists of measurable functions $f$ on $\D$ with
	 	 \begin{align*}
	 	 	\|f\|_{T_{\infty}^{q}}=\left\{\int_{\T} \left(\esssup_{z\in \Gamma_{1/2}(\xi)} |f(z)|\right)^{q}\  |d\xi|\right\}^{1/q},\quad\text{ if } q<+\infty\,.
	 	 \end{align*}
 	  It is known that the definition is independent of the type of the non-tangential region we use.
 	
	 When $q=+\infty$ and $p<+\infty$, the tent space $T_{p}^{\infty}$ consists of measurable functions $f$ on $\D$ with
	 	 \begin{align}\label{TentC}
	 	 	\|f\|_{T_p^\infty}=\sup_I \left(\frac{1}{|I|}\int_{S(I)}|f(z)|^p  \ dA(z) \right)^{1/p}<+\infty.
	 	 \end{align}
	 	 Obviously, an equivalent way to state (\ref{TentC}) is to say that the measure
	 	 $d\mu(z)=|f(z)|^p\,dm(z)$
	 	 is a 1-CM.

	Here, we are interested in the holomorphic version of the tent spaces. We denote that  as 
	$$AT_p^{q} = T_p^q\cap \mathcal{H}(\D).$$
			 The limit case $p=\infty$ corresponds to the Hardy spaces, that is, $AT_{\infty}^q\equiv H^q \,,$ 
	 since the membership of an $f\in H^q$ can be determined by the $q$-integrability over $[0,2\pi)$ of the non-tangential maximal function
	 $$
	 N_{C}f(e^{i\theta})=\esssup_{z\in \Gamma_{C}(e^{i\theta})} |f(z)|\,
	 $$
	of the function $f$ (see \cite{Du}).

It is worth noting the norm for the analytic tent spaces $AT_{p}^{q}$, $1\leq p,q<\infty$, can be expressed equivalently using both radial integrals (see, e.g., \cite[Proposition 3.1, p. 15]{AguilarGal})
\begin{equation}\label{RMpq}
\|f\|_{AT_p^q}\asymp \left(\int_{0}^{2\pi} \left(
\int_{0}^{1} |f(re^{i\theta})|^p\ dr
\right)^{q/p}\frac{d\theta}{2\pi}\right)^{1/q}\doteqdot \rho_{p,q}(f)
\end{equation}
and by expressions involving the derivative (see \cite[Theorem 2, p. 9]{Perala})
$$
\|f\|_{AT_p^q}\asymp \left(\int_{\T}\left(
\int_{\Gamma_{1/2}(\xi)} |f'(z)|^p (1-|z|)^p\frac{dA(z)}{1-|z|}
\right)^{q/p}\ |d\xi|\right)^{1/q}.
$$
The latter also applies to tent spaces $AT_{p}^{\infty}$, $1\leq p<\infty$, obtaining the corresponding equivalent expression (see \cite[p. 321]{CohnVerbitsky})
\begin{equation}\label{TentCMDer}
 \|f\|_{AT_{p}^{\infty}}^p\asymp
|f(0)|^p+\sup_I \frac{1}{|I|}\int_{S(I)}|f'(z)|^p (1-|z|)^p dA(z). 
\end{equation}

 According to \cite{Perala}, if $p,q \in (1,\infty)$ then the $AT^q_p$ equipped with the norm 
(\ref{TentA}) are Banach spaces.
In account of the  duality relation, given by the inner product of $A^2$
$$
<f,g>=\int_{\D} f(z) \overline{g(z)} dA(z),
$$
it holds  that 
\begin{equation}\label{dual}
	( AT^q_p)^* \simeq AT^{q^{'}}_{p^{'}}\,
\end{equation}
(see \cite{Perala,AgConRodr,AguilarGal}). In addition, it will be very useful to us the following estimation. Let $p,q \in (1,\infty)$, $\alpha \in \D$ and $\beta > 1/p + 1/q $ then 
\begin{equation}\label{estimation}
\rho_{p,q}((1-\overline{a}z)^{-\beta})\asymp (1-|\alpha|)^{1/p+1/q-\beta} 
\end{equation}
(see \cite[Example 2.3]{AgConRodr}).

Finally, the next lemma can be deduced by following the proof of \cite[Proposition 2.8]{AgConRodr},  which will be used repeatedly throughout this paper, and describes the growth of such a function $f \in AT_p^q$.
\begin{lemma}
    \label{growth}
    Let $0<p,q<\infty$ and $f \in AT_p^q$. Then $$|f(z)|=o\left(\left(1-|z|\right)^{-1/p-1/q}\right), \quad |z|\to 1^-.$$
\end{lemma}

\subsection{Tent spaces of sequences}
Let $\beta(z,w)$ be the hyperbolic metric on $\D$ and let $D(z,r)=\{w\in\D\ :\ \beta(z,w)<r\}$ be the hyperbolic disc of radius $r>0$ centered at $z\in\D$. The sequence $Z=\{z_k\}$ is a separated sequence if there is a constant $\delta>0$ such that $\beta(z_k,z_j)\geq \delta$ for $j\neq k$. The sequence $Z=\{z_k\}$ is said to be an $(r,\kappa)$-lattice (in the hyperbolic distance), for $r>\kappa>0$, if  
\begin{enumerate}
	\item $\D=\bigcup_{k} D(z_k,r)$,
	\item the sets $D(z_k,\kappa)$ are pairwise disjoint,
\end{enumerate}
Notice that any $(r,\kappa)$-lattice is a separated sequence. 

\begin{remark}\label{remarkestimatedistbound}	
	It is known that if $z,w\in\D$ such that $\beta(z,w)<r$, then there is a constant $C=C(r)>0$ such that
	$\frac{1}{C}(1-|z|)\leq 1-|w|\leq C(1-|z|)$. 
\end{remark}

A standard way to get an $(r,\kappa)$-lattice is through the Whitney Decomposition of the unit disc. This is the $(r,\kappa)$-lattice described in the following example and the one that Luecking makes use of in \cite{Luecking}.  

\begin{example}\label{LueckingCenters}
Let $Z=\{z_{n,k}\}$ be the sequence formed by the centers of the regions 
$$R_{n,j}=\left\{z\in\D\ :\ 1-\frac{1}{2^{n}}\leq |z|<1-\frac{1}{2^{n+1}}\  ,\arg(z)\in\left[\frac{2\pi j}{2^{n}},\frac{2\pi (j+1)}{2^{n}}\right)\right\}$$
for $n\in\N\cup\{0\}$ and $j=0,1,\dots, 2^{n}-1$.  We will refer to them as Luecking regions.
There exist $r>\kappa>0$ such that $Z=\{z_{n,j}\}$ is an $(r,\kappa)$-lattice.

\end{example}

Our approach  in order to get our results is mainly based on the discrete version of the tent spaces. These are the tent spaces of sequences.
\begin{definition} Let $Z=\{z_{n}\}$ be a $(r,\kappa)$-lattice and $0<p,q<+\infty$.
	We say that $\{\lambda_n\}\in T_{p}^{q}(Z)$ if
	\begin{align}\label{TentsequenceA}
	\|\{\lambda_n\}\|_{T_{p}^{q}(Z)}:=\left(\int_{\T} \left(\sum_{z_n\in \Gamma_{1/2}(\xi)} |\lambda_n|^{p}\right)^{q/p}\ |d\xi|\right)^{1/q}<+\infty,
	\end{align}
	the sequence $\{\lambda_n\}_{n}\in T_{\infty}^{q}(Z)$ if
	\begin{align}\label{TentsequenceB}
	\|\{\lambda_n\}\|_{T_{\infty}^{q}(Z)}:=\left(\int_{\T} \left(\sup_{z_n\in \Gamma_{1/2}(\xi)} |\lambda_n|\right)^{q}\ |d\xi|\right)^{1/q}<+\infty,
	\end{align}
	and $\{\lambda_n\}\in T_{p}^{\infty}(Z)$ if
	\begin{align}\label{TentsequenceC}
	\|\{\lambda_n\}\|_{T_{p}^{\infty}(Z)}=\sup_I \left(\frac{1}{|I|}\sum_{z_{n}\in S(I)} |\lambda_n|^{p}(1-|z_n|^2)\right)^{1/p}<+\infty.
	\end{align}
\end{definition}
\begin{remark}\label{independetregionsequence}
	Lemma~\ref{estimate 1}  justifies the independence of the definition  \eqref{TentsequenceA} for any $\Gamma_M (\xi)$.  
\end{remark}

The following lemma is a well-know result in the theory of tent spaces of sequences and it will be useful in order to obtain equivalent expressions when we consider different non-tangential regions. 
\begin{letterlemma}\label{lemma_sup} Let $M_1>M>1/2$. There is a constant $C>0$ such that
	\begin{align*}
	\int_{\T} \left(\sup_{z_k\in \Gamma_{M_1}(\xi)} |\lambda_{k}|^{q}\right)\ |d\xi|\leq C \int_{\T} \left(\sup_{z_k\in \Gamma_{M}(\xi)} |\lambda_{k}|^{q}\right)\ |d\xi|
	\end{align*}
	for all $0<q<+\infty$, and for all sequences $\{z_k\}\subset \D$ and $\{\lambda_{k}\}\subset \C$.
\end{letterlemma}

\subsection{Carleson type measures for tent spaces}

Let $\mathbb{X}$ be a quasi-Banach space for analytic functions in the unit disc. For $s>0$, we will say that a positive Borel measure $\mu$ on $\D$ is $(s,\mathbb{X})$-Carleson measure if there is a positive constant $C$ such that
$$
\int_{\D} |f(w)|^s\ d\mu(w)\leq C \|f\|_{\mathbb{X}}^{s}
$$
for all $f\in\mathbb{X}$. Carleson introduced this concept in his work \cite{Carleson} about the theory of interpolating sequences for Hardy spaces. He proved that the geometric condition \eqref{1CM on D} describes the $(p,H^p)$-Carleson measures.

Following this, we will summarize some results in regard to the Carleson measures for tent spaces. In 1993, Luecking proved in \cite{Luecking} the following result in the setting of tent spaces of analytic functions defined on the upper half-plane. In any case, his argument can be adapted for tent spaces of holomorphic functions on the unit disc.

\begin{lettertheorem}\label{LAG} {\cite[Theorem 3, p. 354]{Luecking}; \cite[Theorem 4.3, p. 23]{AguilarGal}} Let $0<s,p,q<+\infty$, $Z=\left\{z_k\right\}$ the $(r,\kappa)$-lattice consisting of the centers of the Luecking regions, and $n\in\N\cup\{0\}$. Let $\mu$ be a finite positive measure on $\D$. Then the following assertions are equivalent.
\begin{enumerate}
	\item There is a constant $C>0$ such that
$$
\left(\int_{\D} |f(w)|^s\ d\mu(w)\right)^{1/s}\leq C \|f\|_{T_p^q}
$$
for all $f\in AT_p^q$.
\item The sequence 
$$
\mu_k:=\mu(R_k)(1-|z_k|)^{-\frac{s}{p}-sn-1},
$$
where $R_k$ is a Luecking region and $z_k$ is the center of this region, satisfies one of the following:
\begin{enumerate}
	\item[(a)] If $s<p,q$, then $\{\mu_k\}\in T_{\frac{p}{p-s}}^{\frac{q}{q-s}}(Z)$
	\item[(b)] If $p\leq s<q$, then $\{\mu_k\}\in T_{\infty}^{\frac{q}{q-s}}(Z)$ \item[(c)] If $q<s$ or $p\leq s=q$, then $\{\mu_k(1-|z_k|)^{1-\frac{s}{q}}\}\in \ell^\infty$
	\item[(d)] If $s=q<p$, then $\{\mu_k\}\in T_{\frac{p}{p-s}}^{\infty}(Z)$
\end{enumerate}
\end{enumerate} 
\end{lettertheorem}
\begin{remark}
The cases stated above as 2(a), 2(b) and 2(c) were covered by Luecking in \cite[Theorem 3, p. 354]{Luecking}. Although the case $0<s=q<p<+\infty$ also appears, it had to be clarified later in \cite[Theorem 4.3, p. 23]{AguilarGal}.
\end{remark}

In addition, Luecking posed in \cite{Luecking} a more general version of the above question. It consists of characterizing the positive Borel measures $\mu$ on $\D$ for which there is a positive constant $C$ such that
$$
\int_{\T}\left( \int_{\Gamma(\xi)} |f(z)|^t\ d\mu(z)\right)^{s/t}\ |d\xi|\leq C \|f\|_{T_p^q}^s
$$
for every $f\in AT_p^q$. Later, two of the authors of this paper solved this question in \cite[Thm 4.1, p. 18]{AguilarGal}.

\begin{lettertheorem}\cite[Thm 4.1, p. 18]{AguilarGal} Let $0<p,q,s,t<+\infty$, $M>1$, $Z=\left\{z_k\right\}$ an $(r,\kappa)$-lattice, and let $\mu$ be a positive measure on $\D$. Then the following assertions are equivalent:
\begin{enumerate}
\item There is a constant $C>0$ such that
$$
\left(\int_{\T} \left(\int_{\Gamma_{1/2}(\xi)} |f(w)|^t\ d\mu(w)\right)^{s/t}\ |d\xi|\right)^{1/s}\leq C \|f\|_{T_p^q}
$$
for all $f\in AT_p^q$.
\item The measure $\mu$ satisfies the following:
\begin{enumerate}
	\item[(a)] If $0<s<q<+\infty$, $0<t<p<+\infty$, then
	$$
	\int_{\T}\left(\sum_{z_k\in \Gamma_{1/2}(\xi)} \left(
	\frac{\mu^{1/t}(\Delta(z_k,r))}{(1-|z_k|)^{1/p}}
	\right)^{\frac{pt}{p-t}}\right)^{\frac{(p-t)qs}{(q-s)pt}}\ |d\xi|<+\infty.
	$$
	\item[(b)] If $0<s<q<+\infty$, $0<p\leq t<+\infty$, then
	$$
	\int_{\T}\left(\sup_{z_k\in \Gamma_{1/2}(\xi)} 
	\frac{\mu^{1/t}(\Delta(z_k,r))}{(1-|z_k|)^{1/p}}
	\right)^{\frac{qs}{q-s}}\ |d\xi|<+\infty.
	$$
	\item[(c)] If $0<q<s<+\infty$, $0<p,t<+\infty$ or $0<q)s<+\infty$, $0<p\leq t<+\infty$, then
	$$
	\sup_{k} \frac{\mu^{1/t}(\Delta(z_k,r))}{(1-|z_k|)^{\frac{1}{p}+\frac{1}{q}-\frac{1}{s}}}<+\infty.
	$$
	\item[(d)] If $0<q=s<+\infty$, $0<t<p<+\infty$, then
	$$
	\sup_{\xi\in\T} \left(\sup_{\xi\in I} \frac{1}{|I|}\sum_{z_k\in S(I)} \left(\frac{\mu^{1/t}(\Delta(z_k,r))}{(1-|z_k|)^{1/p}}\right)^{\frac{pt}{p-t}}(1-|z_k|)\right)^{\frac{p-t}{pt}}<+\infty.
	$$
\end{enumerate}
	\end{enumerate}
\end{lettertheorem}

\section{Proofs}

\subsection{Proof of Proposition \ref{test function}}
Let us see that $f(z)=\frac{1}{(1-z)\log\left(\frac{e}{1-z}\right)}$ belongs to $AT^q_p$ when $\frac{1}{p}+\frac{1}{q}= 1$ and $p>2$. We will make use of the equivalent norm stated in (\ref{RMpq}).
We start by splitting up the functions as follows.
$$
f(z)=f(z) \chi_{D_1}(z)+f(z) \chi_{D_1^{c}}(z),
$$
where $D_1:=\left\{z\in \D : |z-1|\leq 1\right\}$. Thus, we have that
$$
\rho_{p,q}(f)\leq  \rho_{p,q}(f  \chi_{D_1}) + \rho_{p,q}(f \chi_{D_1^{c}}) \leq \rho_{p,q}(f  \chi_{D_1}) + \frac{1}{2\log\left(\frac{e}{2}\right)}.
$$

Splitting the integrals, we obtain that

\begin{align*}
	&\rho_{p,q}^{q} (f  \chi_{D_1})\leq 2 \int_{0}^{\pi/2}\left(\int_{0}^{1} \frac{\chi_{D_1}(re^{i\theta})\ dr}{|1-re^{i\theta}|^p \log^{p}\left(\frac{e}{|1-re^{i\theta}|}\right)}\right)^{q/p}\ \frac{d\theta}{2\pi}\\
	&\leq 2 \int_{0}^{1}\left(\int_{0}^{1} \frac{\chi_{D_1}(re^{i\theta})\ dr}{|1-re^{i\theta}|^p \log^{p}\left(\frac{e}{|1-re^{i\theta}|}\right)}\right)^{q/p}\ \frac{d\theta}{2\pi}+2 \int_{1}^{\pi/2}\left(\int_{0}^{1} \frac{\chi_{D_1}(re^{i\theta})\ dr}{|1-re^{i\theta}|^p \log^{p}\left(\frac{e}{|1-re^{i\theta}|}\right)}\right)^{q/p}\ \frac{d\theta}{2\pi}\\
	&\leq 2 \int_{0}^{1}\left(\int_{0}^{1} \frac{\chi_{D_1}(re^{i\theta})\ dr}{|1-re^{i\theta}|^p \log^{p}\left(\frac{e}{|1-re^{i\theta}|}\right)}\right)^{q/p}\ \frac{d\theta}{2\pi}+2 \left(\frac{\pi}{\sqrt{3}}\right)^{q} \int_{1}^{\pi/2}\left(\int_{0}^{1} \frac{\chi_{D_1}(re^{i\theta})\ dr}{\theta^p \log^{p}\left(\frac{e\pi}{\sqrt{3} \theta}\right)}\right)^{q/p}\ \frac{d\theta}{2\pi}\\
	&\leq 2 \int_{0}^{1}\left(\int_{0}^{1-\theta} \frac{\chi_{D_1}(re^{i\theta})\ dr}{|1-re^{i\theta}|^p \log^{p}\left(\frac{e}{|1-re^{i\theta}|}\right)}+\int_{1-\theta}^{1} \frac{\chi_{D_1}(re^{i\theta})\ dr}{|1-re^{i\theta}|^p \log^{p}\left(\frac{e}{|1-re^{i\theta}|}\right)}\right)^{q/p}\ \frac{d\theta}{2\pi}\\
	&\quad\quad +  \frac{\pi^q}{(\sqrt{3})^q \log^{q}\left(\frac{e\pi}{\sqrt{3}}\right)}.
\end{align*}

One can show that $|1-re^{i\theta}| \geq \frac{2}{\pi} \sqrt{(1-r)^2+r \theta^2}$ for $r\in [0,1)$ and $\theta\in[0,\pi/2]$. In addition, it is easy to show that
\begin{equation}\label{eq1}
	|1-re^{i\theta}|\geq 
	\begin{cases}
		\frac{\sqrt{3}}{\pi} (1-r),& \text{if}\quad r<1-\theta,\\
		\frac{\sqrt{3}}{\pi} \theta,& \text{if}\quad r\geq 1-\theta.\\
	\end{cases}
\end{equation}

Using the previous inequality \eqref{eq1} and the fact that the function $g(x)=x\log\left(\frac{e}{x}\right)$ is increasing for $0<x<1$, we proceed by estimating the remaining integrals as follows.
\begin{align*}
	&\int_{0}^{1}\left(\int_{0}^{1-\theta} \frac{\chi_{D_1}(re^{i\theta})\ dr}{|1-re^{i\theta}|^p \log^{p}\left(\frac{e}{|1-re^{i\theta}|}\right)}+\int_{1-\theta}^{1} \frac{\chi_{D_1}(re^{i\theta})\ dr}{|1-re^{i\theta}|^p \log^{p}\left(\frac{e}{|1-re^{i\theta}|}\right)}\right)^{q/p}\ \frac{d\theta}{2\pi}\\
	&\leq \left(\frac{\pi}{\sqrt{3}}\right)^q \int_{0}^{1}\left(\int_{0}^{1-\theta} \frac{\chi_{D_1}(re^{i\theta})\ dr}{(1-r)^p \log^{p}\left(\frac{e\pi}{\sqrt{3}(1-r)}\right)}+\int_{1-\theta}^{1} \frac{\chi_{D_1}(re^{i\theta})\ dr}{\theta^p \log^{p}\left(\frac{e\pi}{\sqrt{3}\theta}\right)}\right)^{q/p}\ \frac{d\theta}{2\pi}\\
	&\leq \left(\frac{\pi}{\sqrt{3}}\right)^q \int_{0}^{1}\left(\int_{0}^{1-\theta} \frac{
		\ dr}{(1-r)^p \log^{p}\left(\frac{e\pi}{\sqrt{3}(1-r)}\right)}+\int_{1-\theta}^{1} \frac{ dr}{\theta^p \log^{p}\left(\frac{e\pi}{\sqrt{3}\theta}\right)}\right)^{q/p}\ \frac{d\theta}{2\pi}.
\end{align*}

Now, splitting the integrals into intervals of the form $[e^{-(n+1)},e^{-n}]$, $n\in\N\cup\{0\}$, and making the appropriate estimates, it follows that
\begin{align*}
	&\left(\frac{\pi}{\sqrt{3}}\right)^q \int_{0}^{1}\left(\int_{0}^{1-\theta} \frac{
		\ dr}{(1-r)^p \log^{p}\left(\frac{e\pi}{\sqrt{3}(1-r)}\right)}+\int_{1-\theta}^{1} \frac{ dr}{\theta^p \log^{p}\left(\frac{e\pi}{\sqrt{3}\theta}\right)}\right)^{q/p}\ \frac{d\theta}{2\pi}\allowdisplaybreaks\\
	&\leq \left(\frac{\pi}{\sqrt{3}}\right)^q \sum_{n=0}^{\infty} \int_{1/e^{n+1}}^{1/e^n}\left(\int_{0}^{1-\theta} \frac{
		\ dr}{(1-r)^p \log^{p}\left(\frac{e\pi}{\sqrt{3}(1-r)}\right)}+ \frac{ 1}{\theta^{p-1} \log^{p}\left(\frac{e\pi}{\sqrt{3}\theta}\right)}\right)^{q/p}\ \frac{d\theta}{2\pi}\allowdisplaybreaks\\
	&\leq \left(\frac{\pi}{\sqrt{3}}\right)^q \sum_{n=0}^{\infty} \int_{1/e^{n+1}}^{1/e^n}\left(\sum_{k=0}^{n}\int_{1-1/e^k}^{1-1/e^{k+1}} \frac{
		\ dr}{(1-r)^p \log^{p}\left(\frac{e\pi}{\sqrt{3}(1-r)}\right)}+ \frac{ 1}{\theta^{p-1} \log^{p}\left(\frac{e\pi}{\sqrt{3}\theta}\right)}\right)^{q/p}\ \frac{d\theta}{2\pi}\allowdisplaybreaks\\
	&\leq \left(\frac{\pi}{\sqrt{3}}\right)^q \sum_{n=0}^{\infty} \int_{1/e^{n+1}}^{1/e^n}\Biggl(\sum_{k=0}^{n} e^{(k+1)(p-1)}\int_{1-1/e^k}^{1-1/e^{k+1}} \frac{
		\ dr}{(1-r) \log^{p}\Bigl(\frac{e\pi}{\sqrt{3}(1-r)}\Bigr)}\\
		&\quad\quad + \frac{ 1}{\theta^{p-1} \log^{p}\left(\frac{e\pi}{\sqrt{3}\theta}\right)}\Biggr)^{q/p}\ \frac{d\theta}{2\pi}\allowdisplaybreaks\\
	&= \left(\frac{\pi}{\sqrt{3}}\right)^q \sum_{n=0}^{\infty} \int_{1/e^{n+1}}^{1/e^n}\Biggl(\sum_{k=0}^{n} \frac{e^{(k+1)(p-1)}}{p-1}\left(\frac{1}{\log^{p-1}\left(\frac{ e\pi e^{k}}{\sqrt{3}}\right)}-\frac{1}{\log^{p-1}\left(\frac{ e\pi e^{k+1}}{\sqrt{3}}\right)}\right)\\
	&\quad \quad+ \frac{ 1}{\theta^{p-1} \log^{p}\left(\frac{e\pi}{\sqrt{3}\theta}\right)}\Biggr)^{q/p}\ \frac{d\theta}{2\pi}.
\end{align*}
Applying the fact that $x^{p-1}-y^{p-1}\leq (p-1) (x-y) \sup_{y\leq t\leq x} t^{p-2}$ for $0\leq y\leq x$, we obtain

\begin{align*}
	&\left(\frac{\pi}{\sqrt{3}}\right)^q \sum_{n=0}^{\infty} \int_{1/e^{n+1}}^{1/e^n}\Biggl(\sum_{k=0}^{n} \frac{e^{(k+1)(p-1)}}{p-1}\left(\frac{1}{\log^{p-1}\left(\frac{ \pi e^{k+1}}{\sqrt{3}}\right)}-\frac{1}{\log^{p-1}\left(\frac{ \pi e^{k+2}}{\sqrt{3}}\right)}\right)\\
	&\quad \quad + \frac{ 1}{\theta^{p-1} \log^{p}\left(\frac{e\pi}{\sqrt{3}\theta}\right)}\Biggr)^{q/p}\ \frac{d\theta}{2\pi}\allowdisplaybreaks\\
	&\leq 
	\left(\frac{\pi}{\sqrt{3}}\right)^q \sum_{n=0}^{\infty} \int_{1/e^{n+1}}^{1/e^n}\Biggl(\sum_{k=0}^{n} \frac{e^{(k+1)(p-1)}}{\log^{p-2}\left(\frac{ \pi e^{k+1}}{\sqrt{3}}\right)}\left(\frac{1}{\log\left(\frac{ \pi e^{k+1}}{\sqrt{3}}\right)}-\frac{1}{\log\left(\frac{ \pi e^{k+2}}{\sqrt{3}}\right)}\right)\\
	&\quad \quad + \frac{ 1}{\theta^{p-1} \log^{p}\left(\frac{e\pi}{\sqrt{3}\theta}\right)}\Biggr)^{q/p}\ \frac{d\theta}{2\pi}\allowdisplaybreaks\\
	&= 
	\left(\frac{\pi}{\sqrt{3}}\right)^q \sum_{n=0}^{\infty} \int_{1/e^{n+1}}^{1/e^n}\Biggl(\sum_{k=0}^{n} \frac{e^{(k+1)(p-1)}}{\log^{p-2}\left(\frac{ \pi e^{k+1}}{\sqrt{3}}\right)}\left(\frac{1}{\log\left(\frac{ e\pi e^{k}}{\sqrt{3}}\right)\log\left(\frac{ \pi e^{k+2}}{\sqrt{3}}\right)}\right)\\
	&\quad \quad + \frac{ 1}{\theta^{p-1} \log^{p}\left(\frac{e\pi}{\sqrt{3}\theta}\right)}\Biggr)^{q/p}\ \frac{d\theta}{2\pi}\allowdisplaybreaks\\ 
	&\leq 
	\left(\frac{\pi}{\sqrt{3}}\right)^q \sum_{n=0}^{\infty} \int_{1/e^{n+1}}^{1/e^n}\left(\sum_{k=0}^{n} \frac{e^{(k+1)(p-1)}}{\log^{p-1}\left(\frac{ \pi e^{k+1}}{\sqrt{3}}\right)}\left(\frac{1}{\log\left(\frac{ \pi e^{k+2}}{\sqrt{3}}\right)}\right)+ \frac{ 1}{\theta^{p-1} \log^{p}\left(\frac{e\pi}{\sqrt{3}\theta}\right)}\right)^{q/p}\ \frac{d\theta}{2\pi}\\
	&\leq \left(\frac{\pi}{\sqrt{3}}\right)^q \sum_{n=0}^{\infty} \int_{1/e^{n+1}}^{1/e^n}\left(\sum_{k=0}^{n} \frac{e^{(k+1)(p-1)}}{\log^{p}\left(\frac{ e\pi e^{k}}{\sqrt{3}}\right)}+ \frac{ 1}{\theta^{p-1} \log^{p}\left(\frac{e\pi}{\sqrt{3}\theta}\right)}\right)^{q/p}\ \frac{d\theta}{2\pi}.
\end{align*}
Since $\log\left(\frac{\pi e^{k+1}}{\sqrt{3}}\right)=\log\left(\frac{e\pi}{\sqrt{3}}\right)+k> (k+1)$, for every $k\in \N\cup\{0\}$, it follows that
\begin{align*}
	&\left(\frac{\pi}{\sqrt{3}}\right)^q \sum_{n=0}^{\infty} \int_{1/e^{n+1}}^{1/e^n}\left(\sum_{k=0}^{n} \frac{e^{(k+1)(p-1)}}{\log^{p}\left(\frac{ \pi e^{k+1}}{\sqrt{3}}\right)}+ \frac{1}{\theta^{p-1} \log^{p}\left(\frac{e\pi}{\sqrt{3}\theta}\right)}\right)^{q/p}\ \frac{d\theta}{2\pi}\allowdisplaybreaks\\
	&< \left(\frac{\pi}{\sqrt{3}}\right)^q \sum_{n=0}^{\infty} \int_{1/e^{n+1}}^{1/e^n}\left(\sum_{k=0}^{n} \frac{e^{(k+1)(p-1)}}{(k+1)^p}+ \frac{1}{\theta^{p-1} \log^{p}\left(\frac{e\pi}{\sqrt{3}\theta}\right)}\right)^{q/p}\ \frac{d\theta}{2\pi}\allowdisplaybreaks\\
	&<  \frac{1}{2\pi} \left(\frac{\pi}{\sqrt{3}}\right)^q \sum_{n=0}^{\infty} \frac{1}{e^{n+1}}\left(\sum_{k=0}^{n} \frac{e^{(k+1)(p-1)}}{(k+1)^p}+ \frac{ e^{(n+2)(p-1)}}{(n+2)^p}\right)^{q/p}\allowdisplaybreaks\\
	&\leq \frac{1}{2\pi} \left(\frac{\pi}{\sqrt{3}}\right)^q \sum_{n=0}^{\infty} \frac{1}{e^{n+1}}\left(\sum_{k=0}^{n+1} \frac{e^{(k+1)(p-1)}}{(k+1)^p}\right)^{q/p}<\infty\,.
\end{align*} 
\\
\subsection{Proof of Proposition \ref{first term}}
Let $1/p + 1/q <1$, $n\in\N$ and an $f(z)=\sum_{k\geq 0}a_k z^k \in AT^q_p$. If $N\in \mathbb N$, then we set
  $$S_N^{(n)}=\sum_{k=0}^N \frac{a_k}{n+k+1}.$$ 
  Thus, for $N>M$,
 \begin{align}
\Big|S_N^{(n)} -S_M^{(n)} \Big| &= \Big|\sum_{k=M}^N \frac{a_k}{n+k+1}\Big|
=2\, \Big|\sum_{k=M}^N  \int_{\mathbb D} f(z)\, \overline{z}^k\,|z|^{2n} \,dA(z) \Big|\notag\\
 & =2\, \Big|  \int_{\mathbb D} f(z)\, |z|^{2n}\, \left(\sum_{k=M}^N\,\overline{z}^k\right) \,dA(z) \Big|\notag\\
 & \leq 2\,   \int_{\mathbb D} |f(z)|\, \,\Big|\overline{\frac{(1-z^{N+1})-(1-z^M)}{1-z}}\Big| \,dA(z). \label{p1.2-ineq1}
 \end{align}
Since
$$
\Big|\frac{(1-z^M)}{1-z} - \frac{1}{1-z}\Big| \to 0,\,\,\, M  \to \infty
$$
uniformly in compact subsets of $\D$, 
$$
  \Big|  \frac{(1-z^{N+1})-(1-z^M)}{1-z} \Big|\to 0\,, \quad M \to \infty\,.
$$
Combined with the fact that 
$$
|f(z)|\, \Big|  \frac{(1-z^{N+1})-(1-z^M)}{1-z} \Big|\lesssim |f(z)|\, \Big|  \frac{1}{1-z} \Big|\,,\quad z\in \D,
$$
\begin{equation}
\int_{\mathbb D} |f(z)|\, \,\Big|\frac{1}{1-z}\Big| \,dA(z) \leq \|f\|_{AT^q_p} 
\Big\|\frac{1}{1-z}\Big\|_{AT^{q'}_{p'}}, \label{p1.2-ineq2}
\end{equation}
where $1/p+1/p'=1, 1/q +1/q'=1$ and that
$$
\frac{1}{1-z} \in {AT^{q'}_{p'}} 
$$
 due to \cite[Example 2.3]{AgConRodr} and (\ref{RMpq}) and using that $1/p'+1/q' >1$, we finally get 
 $$
 |S_N^{(n)} -S_M^{(n)}| \to 0 \,, \quad M \to \infty\,.
 $$
 
 In addition, notice that arguing as in \eqref{p1.2-ineq1} and using inequality \eqref{p1.2-ineq2}, we obtain that there is a positive constant $C=C(p,q)$ such that $$\sup_{n\in\N}\left| S_N^{(n)}\right|\leq C \|f\|_{AT_p^q}$$ for every $N\in\N$. Therefore, by using the convergence of the parcial sums, we get that
 $$
 \sup_{n}\left|\sum_{k=0}^{\infty} \frac{a_k}{n+k+1}\right|\leq C \|f\|_{AT_p^q}.
 $$
 \\
 \subsection{Proof of Proposition \ref{Hmu=Imu}}
 Let  $\mu$ be an $1-$CM or equivalently let (\ref{growth of the moments}) be true.  According to \cite[Theorem 3.3]{AgConRodr}, if $f(z)=\sum_{k\geq 0} a_k z^k\in AT^q_p,$ $\frac{1}{p}+\frac{1}{q}<1,$ \,$ p>2$ and $ q>1$, then  there exists a $q_{0} >2$ such that 
 $$
 AT^q_p \subset A^{q_{0}}\,.
 $$
 Therefore, for each $n\in \N$  the series 
 $$
 \sum_{k\geq 0} \mu_{n,k} a_k
 $$
 converges absolutely since 
 $$
 \sum_{k\geq 0}  \mu_{n,k} |a_k| = \sum_{k\geq 0} \int_0^1 t^{n+k} \,d\mu(t) |a_k| \leq C \sum_{k\geq 0} \frac{|a_k|}{k+1} \leq C'\|f\|_{A^{q_{0}}},
 $$
 where in the last inequality we appeal to the strength of (\ref{hardy's inequality Bergman}). Here, $C$ and $C'$ are positive constants. Consequently, not only does the series 
 $$
 \sum_{n=0}^{\infty}\left(\sum_{k=0}^{\infty}
 \mu_{n+k} a_k \right)z^n, \quad z\in \mathbb D\,,
 $$
 represents an analytic function in the unit disc but relation (\ref{identity Hm}) holds as well. 
 
\subsection{Proof of Proposition \ref{well definition of Imu}}
Assume that there is a constant $C>0$ such that  
\begin{equation}\label{Lueck}
\int_0^1 |f(t)| \,d\mu(t) \leq C \|f\|_{AT^q_p} \,,\quad f\in AT^q_p\,.
\end{equation}
Fixing a $z\in \D$,  
$$
\int_0^1 \frac{|f(t)|}{|1-tz|} \,d\mu(t) \leq \frac{1}{1-|z|} \int_0^1 |f(t)| \,d\mu(t)\leq \frac{C}{1-|z|}\|f\|_{AT^q_p}\,.
$$
So, (\ref{Lueck}) implies that the $\mathcal I_{\mu}$ is well defined in the spaces of our interest. Thinking the other way around, if we assume that the integral (\ref{Im(f)}) is well defined for any $ f\in  AT^q_p$ and for any $z\in \D$, then we set $z=0$. It holds that 
$$
\int_0^1 |f(t)| \,d\mu(t) < \infty\,.
$$
That and an application of the uniform boundedness principal on the family of operators 
$$
T_{s} \, : \,AT^q_p \rightarrow L^1([0,1), \mu) ,
$$ 
where $T_{s}(f)= f\, \chi_{\{0<|z|<s\}}\,, s\in[0,1),$
leads to the strength of (\ref{Lueck}). As a consequence, the well definition of the $\mathcal I_{\mu}$ is equivalent to (\ref{Lueck}).
The positive Borel measures on $\D$ that satisfy the inequality 
\begin{equation*}
	\int_{\D} |f(z)| \,dA(z) \leq C \|f\|_{AT^q_p} \,,\quad f\in AT^q_p\, ,
\end{equation*}
have been characterized in Theorem~\ref{LAG}.
Adapting that result to our setting, we get that (\ref{Lueck}) is equivalent to the 
condition \eqref{final}.
Moreover, under the condition \eqref{final} and after having fixed a $z\in \D$, it is true that 
\begin{equation*}
\I_{\mu}(f)(z)=\int_0^1\frac{f(t)}{1-tz}d\mu(t)=\sum_{n\geq 0}  \left(\int_0^1 f(t) t^n \,d\mu(t)\right) z^n \,,\quad f\in AT^q_p
\end{equation*}
and that $\I_{\mu}(f) \in \H(\D)$\,.

\subsection{Proof of Theorem  $1.5$.}
\ \newline
\begin{Prf}\textit{Theorem \ref{Theorem I}(i)}
First, assume that $\mathcal{I}_{\mu}: AT_p^q\rightarrow AT_p^q$ is bounded. Consider the family of functions
 $$
f_\alpha(z)=\frac{(1-\alpha)^{\frac{1}{p'}+\frac{1}{q'}}}{(1-\alpha z)^2},\quad z \in \D
$$ where $\alpha \in [0,1)$ and $1/p'=1-1/p, 1/q'=1-1/q$\,. Then  $2>\frac{1}{p'}+\frac{1}{q'}>1>\frac{1}{p}+\frac{1}{q}$. 
So, in view of (\ref{RMpq}) and (\ref{estimation}), we get
$$
\sup_{\alpha \in [0,1)}\|f_{\alpha}\|_{AT^q_p} \lesssim 1\,.
$$

The assumption of the boundedness of $\mathcal{I}_{\mu}$ and the duality (\ref{dual}) imply that 
  \begin{equation*}
 \|\mathcal{I}_\mu(f_\alpha) \|_{AT_{p}^{q}}=\sup_{\| g \|_{AT_{p'}^{q'}}\leq 1} \left|\int_{\D} \overline{g(w)} \mathcal{I}_\mu(f_\alpha)(w)\ dA(w) \right| \lesssim 1
  	\end{equation*}
  for every $\alpha \in [0,1)$.
  
  Now, we set as $g_{\alpha}(z)=\frac{(1-\alpha)^{\frac{1}{p}+\frac{1}{q}}}{(1-\alpha z)^2}$, still  $\alpha \in [0,1)$. According to (\ref{estimation}) these belong to  the unit ball of $AT^{q'}_{p'}$. 
  Thus, 
  \begin{equation*}
  	\begin{split}
  		  		\int_{\D} \overline{g_{\alpha}(z)}\, \mathcal{I}_\mu(f_\alpha)(z)\ dA(z) & =
  		  		\int_{0}^1 \frac{1}{\pi} \int_0^{2\pi}\overline{g_{\alpha}(re^{i\theta})}\, \int_0^1 \frac{f_{\alpha}(t)}{1-tre^{i\theta}} \,d\mu(t)\ d\theta\, r dr \\
  		  		& =\int_{0}^1 \int_{0}^{1} f_{\alpha}(t) \frac{1}{\pi} \int_0^{2\pi} \frac{\overline{g_{\alpha}(re^{i\theta})}}{1-tre^{i\theta}}\,d\theta \,d\mu(t)\, r dr\\
  		  		& =  \int_{0}^1 \int_{0}^{1} f_{\alpha}(t) 
  		  	\overline{\left( \frac{1}{\pi} \int_{0}^{2\pi}
  		  		 \frac{g_{\alpha}(re^{i\theta})}{1-tre^{-i\theta}}\ d\theta\right)} \,d\mu(t)\, r dr\\
  	  		 &= 2 \int_{0}^1 \int_{0}^{1} f_{\alpha}(t) g_{\alpha}(r^2 t) d\mu(t) r dr\,.
  		\end{split}
  	\end{equation*}
  The positivity of the integrand allows a change in the order of integration in the last representation. Consequently, for $\alpha \in [0,1)$, we have
  \begin{equation*}
  	\begin{split}
  		\int_{\D} \overline{g_{\alpha}(z)}\, \mathcal{I}_\mu(f_\alpha)(z)\ dA(z) & = 2 \int_{0}^1 f_{\alpha}(t)  \int_{0}^{1}  g_{\alpha}(r^2 t) r dr\, d\mu(t)\,.
  \end{split}
\end{equation*}


Combining the above, it follows that
\begin{align*}
1\gtrsim &\,2  \int_0^1 f_\alpha(t) \int_0^1 \ g_\alpha(r^2 t)\ r dr \ d\mu(t) \,=\,\int_{0}^{1} \frac{(1-\alpha)^{2}}{(1-\alpha t)^2} \int_{0}^{1} \frac{2r}{(1-\alpha r^2 t)^2}\ dr \ d\mu(t)\\
&= \,\int_{0}^{1} \frac{(1-\alpha)^{2}}{(1-\alpha t)^2} \int_{0}^{1} \frac{1}{(1-\alpha r t)^2}\ dr \ d\mu(t) \geq \,\int_{0}^{1} \frac{(1-\alpha)^{2}}{(1-\alpha t)^2} \int_{\alpha}^{1} \frac{1}{(1-\alpha r t)^2}\ dr \ d\mu(t) \\
& \geq  \,\int_{0}^{1} \frac{(1-\alpha)^{3}}{(1-\alpha^2 t)^4} \ d\mu(t) \geq  \int_\alpha^1 \frac{(1-\alpha)^{3}}{ (1-\alpha^2 t)^{4}}\ d\mu(t) \gtrsim \frac{\mu([\alpha,1))}{1-\alpha}.
\end{align*}
Hence, $\mu$ is an $1$-CM.

Now let us assume that $\mu$ is $1$-CM. Then, the operator $\mathcal{I}_\mu$ is well defined on $AT^q_p$.
To prove that $\mathcal{I}_\mu $ is bounded on $AT^q_p$, it is enough to show that
$$
|\langle \mathcal{I}_\mu(f),g \rangle|\lesssim \|f\|_{AT^q_p} \|g\|_{AT^{q'}_{p'}}
$$
for every $f\in AT^q_p$ and $g\in AT^{q'}_{p'}$. So,
\begin{align*}
|\langle \mathcal{I}_\mu(f),g \rangle|&= \left|\int_{\D} \mathcal{I}_{\mu}(f)(z)\overline{g(z)} d A(z)\right|= 2 \left|\int_0^1 \int_0^1 f(t)\ \overline{g(r^2 t)}\ d\mu(t)\ r dr\right|\\
&\leq  \int_0^1 |f(t)| \int_0^1 |g(r^2 t)|\ rdr\ d\mu(t)\\
&= \int_0^1 |f(t)| G(t)\ d\mu(t), 
\end{align*}
where $G(t):=\int_0^1 |g(r^2 t)|\ rdr$. Applying Hölder's inequality, we get
\begin{align*}
\int_0^1 |f(t)| G(t)\ dt&\leq \left(\int_0^1 |f(t)|^q (1-t)^{q/p}\ d\mu(t)\right)^{1/q}\left(\int_0^1 |G(t)|^{q'} (1-t)^{{- q'/p}}\ d\mu(t)\right)^{1/q'}\\
&=:I_1\cdot I_2.
\end{align*}
Now, we proceed with the estimations of $I_1$ and $I_2$.
By Theorem~\ref{LAG}, we have that $I_1\lesssim \|f\|_{AT_p^q}$ if and only if the sequence $
\mu_{n,j}=2^{n}\mu(R_{n,j}\cap [0,1))
$ satisfies the following conditions, where $R_{n,j}$ denotes the Luecking regions (see Example \ref{LueckingCenters}):
\begin{enumerate}
	\item[(1)] If $p\leq q$, then $\{\mu_{n,j}\}\in\ell^\infty$.
	\item[(2)] If $q<p$, then $\{\mu_{n,j}\}\in T_{\frac{p}{p-q}}^{\infty}(Z)$.
\end{enumerate}
 By using $\mu$ is $1$-Carleson measure, let us see that the previous conditions are satisfied:
\begin{enumerate}
	\item[(1)] If $p\leq q$, then $
	\mu_{n,j}\lesssim 2^{n} \frac{1}{2^n}=1.$
	\item[(2)] If $q<p$, then
	\begin{align*}
	&\sup_{\xi\in \T, h\in (0,1)} \frac{1}{h} \sum_{z_{n,j}\in D(\xi,h)\cap \D} \frac{|\mu_{n,j}|^{\frac{p}{p-q}}}{2^{n}}  \lesssim  \sup_{k_0\geq 0} 2^{k_0} \sum_{k\geq k_0} \frac{1}{2^{k}}=2<\infty.
	\end{align*}
\end{enumerate}
Therefore, we get that $I_1\lesssim\|f\|_{AT_p^q}$.

Now, let us show that $I_2\lesssim \|g\|_{AT_{p'}^{q'}}$. First, we split $I_2$ into the follows two integrals that we will estimate separately.
\begin{align*}
I_2^{q'}\leq \int_0^{1/2} |G(t)|^{q'} (1-t)^{-q'/p}\ d\mu(t) + \int_{1/2}^1 |G(t)|^{q'} (1-t)^{-q'/p}\ d\mu(t)=: I_{2,1}+I_{2,2}.
\end{align*}

Obviously, it holds that $I_{2,1}\lesssim \|g \|_{AT_{p'}^{q'}}^{q'}$. In addition,
\begin{align*}
I_{2,2}&\leq  \int_{1/2}^{1} (1-t)^{-q'/p} \left(\int_0^t |g(u)|\ du\right)^{q'}\ d\mu(t)\\
&=  \int_{1/2}^{1} (1-t)^{-q'/p} \left(\int_{1-t}^1 |g(1-u)|\ du\right)^{q'}\ d\mu(t)\\
&\leq  \int_{0}^{1} (1-t)^{-q'/p} \left(\int_{1-t}^1 |g(1-u)|\ du\right)^{q'}\ d\mu(t)\\
&=: \int_{0}^{1} H(t)\ d\mu(t).
\end{align*}
Notice that, for $\gamma\in [0,1)$, using integration by parts we have
\begin{align*}
\int_{0}^{1} H(t)\ d\mu(t)=H(0)\mu([0,1)) -\lim\limits_{t\rightarrow 1^{-}} H(t)\mu([t,1))+ \int_{0}^{1} \mu([s,1))H'(s)\ ds.
\end{align*} 
By means of Lemma \ref{growth}, we establish  $\int_{0}^{t} |g(x)|\ dx= \text{o}\left((1-t)^{1-\frac{1}{p'}-\frac{1}{q'}}\right)$ as $t\rightarrow {1}^{-}$, which together with the fact that $\mu$ is a $1$-Carleson measure, this leads to the following 
\begin{align*}
\lim\limits_{t\rightarrow 1^{-}} H(t)\mu([t,1))\lesssim \lim\limits_{t\rightarrow 1^{-}} H(t) (1-t)= \lim\limits_{t\rightarrow 1^{-}} (1-t)^{1-\frac{q'}{p}} \left(\int_{0}^{t} |g(x)|\ dx\right)^{q'}=0.
\end{align*} Therefore,
\begin{align*}
	\int_{0}^{1} H(t)\ d\mu(t)&=\int_{0}^{1} \mu([s,1))H'(s)\ ds\lesssim\int_{0}^{1} (1-s)H'(s)\ ds\\
	&=\lim\limits_{t\rightarrow 1^{-}} (1-t)H(t)-H(0)+\int_{0}^{1} H(s)\ ds=\int_{0}^{1} H(s)\ ds\\
	&=\int_{0}^{1}(1-t)^{-q'/p} \left(\int_{1-t}^1 |g(1-u)|\ du\right)^{q'}\ dt.
\end{align*} 
By applying classical Hardy's inequality \cite[p. 244-245]{HLP}, we get that
\begin{align*}
I_{2,2}\lesssim\int_{0}^{1}(1-t)^{-q'/p} \left(\int_{1-t}^1 |g(1-u)|\ du\right)^{q'}\ dt\leq \left(\frac{q'}{1-\frac{q'}{p}}\right)^{q'}\int_{0}^{1} |g(t)|^{q'} (1-t)^{q'-\frac{q'}{p}}\ dt.
\end{align*}
By Theorem~\ref{LAG}, we have that $I_{2,2}\lesssim \|g\|_{AT_{p'}^{q'}}^{q'}$ if and only if the sequence
\begin{align*}
\lambda_{n,j}=|R_{n,j}\cap [0,1)|\left(\frac{1}{2^n}\right)^{q'/p'}\left(\frac{1}{2^n}\right)^{-\frac{q'}{p'}-1}\asymp 1
\end{align*}
satisfies the following conditions:
\begin{enumerate}
	\item[(1')] If $p'> q'$, then $\{\lambda_{n,j}\}\in T_{\frac{p'}{p'-q'}}^{\infty}(Z)$.
	\item[(2')] If $p'\leq q'$, then $\{\lambda_{n,j}\}\in \ell^{\infty}$.
\end{enumerate}
By using that $\mu$ is a $1$-Carleson measure, one can check easily that the previous conditions are satisfied:
\begin{enumerate}
	\item[(1')] If $q'<p'$, then
	\begin{align*}
		&\sup_{\xi\in \T, h\in (0,1)} \frac{1}{h} \sum_{z_{n,j}\in D(\xi,h)\cap \D} \frac{|\lambda_{n,j}|^{\frac{p'}{p'-q'}}}{2^{n}} \asymp \sup_{k_0\geq 0} 2^{k_0} \sum_{k\geq k_0} \frac{1}{2^{k}}
		=\sup_{k_0\geq 0} 2^{k_0} \frac{1/2^{k_0}}{1-(1/2)}=2<\infty.
	\end{align*}
		\item[(2')] If $p'\leq q'$, then obviously $\{\lambda_{n,j}\}\in \ell^{\infty}$.
\end{enumerate}
 Hence, we conclude that $\mathcal I_\mu$ is bounded on $AT_p^q$.
\end{Prf}
\ \newline

\begin{Prf}\textit{Theorem \ref{Theorem I}(ii)}
First, assume  $\mathcal{I}_{\mu}: AT_p^q\rightarrow AT_p^q$ is compact. 
Let us consider a sequence of real numbers $\{a_n\} \subset [0,1)$ such that $a_n\to 1$ as $n\to \infty$ and the family of test functions defined as  $$f_n(z)=\frac{(1-a_n)^{\frac{1}{p'}+\frac{1}{q'}}}{(1-a_n z)^2}.$$ 
Since $2>\frac{1}{p'}+\frac{1}{q'}> 1 >\frac{1}{p}+\frac{1}{q}$, it follows that
 $$\sup_{n}\|f_{n} \|_{AT_p^q}\lesssim 1 \,.$$ 
 Moreover, it is clear that $f_{n}\to 0, n \to \infty,$ uniformly on compact subsets of $\D$. Then, \cite[Lemma 3.7]{Tjani} yields 
 $$\lim\limits_{n\to\infty} \|\mathcal{I}_{\mu}(f_{n}) \|_{AT_p^q} =0\, .$$ This means that for each $\varepsilon >0$ there exists a $n_0\in \N$ such that  $\|\mathcal{I}_{\mu}(f_{n}) \|_{AT_p^q} <\varepsilon$, for every $n\geq n_0$. 
For such an $n$, mimicking the argument applied in the necessity part of Theorem \ref{Theorem I}(i) we obtain $$\frac{\mu([a_n,1))}{1-a_n}\lesssim \|\mathcal{I}_{\mu}(f_{n}) \|_{AT_p^q} \lesssim \varepsilon,$$ so that $\mu $ is an $1-$VCM.

Reciprocally, let us assume that $\mu $ is a Vanishing $1-$Carleson measure. Consider $\{f_{n}\}_{n=0}^{\infty}\subset AT_p^q$ such that $\sup\limits_{n\in \N}\|f_n \|_{AT_p^q} <\infty$ and $f_n\to 0$ uniformly on compact subsets of $\D$. We  will prove that $\lim\limits_{n\to\infty} \|\mathcal{I}_{\mu}(f_n) \|_{AT_p^q} =0$.
Let $\varepsilon>0$.
Arguing as in the sufficiency part of Theorem \ref{Theorem I}(i), 
\begin{align*}
\|\mathcal{I}_{\mu}(f_n)\|_{AT_p^q}&= \sup\limits_{\|g\|_{AT_{p'}^{q'}}\leq 1}|\langle \mathcal{I}_\mu(f_n),g \rangle|\\& \leq \sup\limits_{\|g\|_{AT_{p'}^{q'}}\leq 1}  \left(\int_0^1 |f_n(t)|^q (1-t)^{\frac{p}{q}}\ d\mu(t)\right)^{1/q}\left(\int_0^1 |G(t)|^{q'} (1-t)^{{-\frac{q'}{p}}}\ d\mu(t)\right)^{1/q'},
\end{align*}
where $G(t):=\int_0^1 |g(r^2 t)|\ rdr$. By following the lines of Theorem \ref{Theorem I}(i), the proof can be reduced to show that 
$$\int_0^1 |f_n(t)|^q (1-t)^{\frac{q}{p}}\ d\mu(t)\lesssim \varepsilon.$$
Since $\mu $ is a Vanishing $1-$Carleson measure, there exists $\eta_0 \in (0,1)$ such that 
\begin{equation}
    \label{vanishing}
 \frac{\mu([t,1))}{1-t}<\varepsilon\text{ for all }\eta_0\leq t<1 .  
\end{equation}
Next, notice that $
    \lim\limits_{ t \to 1^{-}} |f_n(t)|^q(1-t)^{\frac{q}{p}} \mu ([t,1))=0,
$
since $ |f_n(t)| =o\left((1-t)^{-\frac{1}{p}-\frac{1}{q}}\right)$ as $t\rightarrow {1}^{-}$ (see Lemma \ref{growth}) and $\mu $ is a Vanishing $1-$Carleson measure. So, there exists $\eta_1\in (0,1)$ such that 
\begin{equation}
    \label{limit}
    |f_n(t)|^q(1-t)^{\frac{q}{p}} \mu ([t,1))<\varepsilon, \text{ for all }t\geq \eta_1.
\end{equation}

Let $\eta=\max\{\eta_0, \eta_1\}$. Then,
$$\int_0^1 |f_n(t)|^q (1-t)^{\frac{q}{p}}\ d\mu(t)=\int_0^{\eta} |f_n(t)|^q (1-t)^{\frac{q}{p}}\ d\mu(t)+\int_{\eta}^1 |f_n(t)|^q (1-t)^{\frac{q}{p}}\ d\mu(t)=:I_1+ I_2.$$
Since $f_n\to 0$ uniformly on compact subsets on $\D$, there exists $n_0\in \N$ such that for all $n\geq n_0$, $\sup\limits_{t \in [0, \eta]}|f_n(t)|<\varepsilon$. Therefore, on the one hand, for $n\geq n_0$,
$$I_1\leq \left(\sup\limits_{t \in [0, \eta]}|f_n(t)|\right)^q \mu([0, \eta])\lesssim\varepsilon.$$
On the other hand, denote $H_n(t)= |f_n(t)|^q (1-t)^{\frac{q}{p}}$. By an integration by parts, \eqref{limit} and \eqref{vanishing}, 
$$I_2\leq \varepsilon + \int_{\eta}^1 H_n'(t)\mu([t,1))dt\leq \varepsilon\left(1+ \int_{\eta}^1 H_n'(t)(1-t)dt\right).$$
By using Lemma \ref{growth} and an integration by parts,
$$ \int_{\eta}^1 H_n'(t)(1-t)dt\lesssim \int_{\eta}^1 |f_n(t)|^q (1-t)^{\frac{q}{p}} dt. $$
By Theorem~\ref{LAG}, we have that $I_{2}\lesssim \varepsilon \sup\limits_k \|f_k\|_{AT_p^q}^q$ if and only if the sequence
\begin{align*}
\lambda_{n,j}=|R_{n,j}\cap [0,1)| 2^{-n\frac{q}{p}}2^{-n\left(-\frac{q}{p}-1\right)}\asymp 1
\end{align*}
satisfies the following conditions:
\begin{enumerate}
	\item[(1)] If $p\leq q$, then $\{\lambda_{n,j}\}\in \ell^{\infty}$. This condition obviously holds. 
	\item[(2)] If $p> q$, then $\{\lambda_{n,j}\}\in T_{\frac{p}{p-q}}^{\infty}(Z)$. Indeed, \begin{align*}
		&\sup_{\xi\in \T, h\in (0,1)} \frac{1}{h} \sum_{z_{n,j}\in D(\xi,h)\cap \D}(1-|z|^2) \asymp \sup_{k_0\geq 0} 2^{k_0} \sum_{k\geq k_0} 2^{-k} =\sup_{k_0\geq 0} 2^{k_0} \frac{2^{-k_0}}{1-(1/2)}=2<\infty.
	\end{align*}
\end{enumerate}
\end{Prf}

\subsection{Proof of Theorem \ref{about Hm}}
\ \newline

\begin{Prf}\textit{the necessary condition for the boundedness of the generalized Hilbert operator $\mathcal H_{\mu}$.}
	
Let $q \in (1,\infty), \, p>2,\, 1/p+1/q <1.$ We assume that $\mathcal H_{\mu}$ is well defined and bounded on the $AT^q_p$ space.
		
		Consider the family of functions $$f_\alpha(z)=\frac{(1-\alpha)^{\frac{1}{p'}+\frac{1}{q'}}}{(1-\alpha z)^2}=(1-\alpha)^{\frac{1}{p'}+\frac{1}{q'}}\sum_{n=0}^{\infty} (n+1) \alpha^n z^n, \quad z \in \D,$$ where $\alpha \in [0,1),$ $1/p^{'} + 1/p =1$ and $ 1/q^{'} + 1/q =1$. Notice that $1/p + 1/q < 1 <  \frac{1}{p'}+\frac{1}{q'} <2.$
				 It holds that 
		 $$\sup_{\alpha \in (0,1)}\|f_{\alpha}\|_{AT^q_p} < \infty\,.$$  
		 Reasoning the same way,
		 $$
		 \|f_{\alpha}\|_{AT^{q^{'}}_{p^{'}}} \lesssim (1-\alpha)^{\frac{2}{p'}+\frac{2}{q'}-2}
		 $$
		 for every $\alpha \in (0,1).$
		 
		In account of the duality described in (\ref{dual}) 
and the  boundedness of $\mathcal H_{\mu}$, we get that
	$$
	\left|\int_{\D} \overline{g(z)} \mathcal{H}_\mu(f_\alpha)(z)\ dA(z) \right|\lesssim \|\mathcal H_{\mu}(f_{\alpha})\|_{AT^q_p} \,\,\|g\|_{AT^{q^{'}}_{p^{'}}} \lesssim 
		\|f_{\alpha}\|_{AT^q_p} \,\,\|g\|_{AT^{q^{'}}_{p^{'}}} \lesssim 
		\,\,\|g\|_{AT^{q^{'}}_{p^{'}}}
	$$
	for any $g \in AT^{q^{'}}_{p^{'}}$.
		In particular, setting 
				$$
				g(z)=f_\alpha(z),
				$$ 
		it follows that for every $\alpha \in (0,1)$
	\begin{align*}
		&(1-\alpha)^{\frac{2}{p'}+\frac{2}{q'}-2}\gtrsim\left|\int_{\D} \overline{f_\alpha(z)} \mathcal{H}_\mu(f_\alpha)(z)\ dA(z)\right|\\
		&\quad =\left|\int_{\D} \overline{\left(\sum_{m=0}^{\infty} (1-\alpha)^{\frac{1}{p'}+\frac{1}{q'}} (m+1)\alpha^m z^m\right)} \left(\sum_{n=0}^{\infty}\left(\sum_{k=0}^{\infty} \mu_{n+k}(1-\alpha)^{\frac{1}{p'}+\frac{1}{q'}} (k+1)\alpha^k\right)z^n\right)\ dA(z)\right|\\
		&\quad =\, (1-\alpha)^{\frac{2}{p'}+\frac{2}{q'}}\, \left|\int_{\D}  \left(\sum_{m,n=0}^{\infty} (m+1) \left(\sum_{k=0}^{\infty} \mu_{n+k} (k+1) \alpha^{k+m}\right)z^n \overline{z}^m\right)\ dA(z)\right|\\
		& \quad =\, (1-\alpha)^{\frac{2}{p'}+\frac{2}{q'}} \left|\int_{0}^1\, \,\frac{1}{\pi}\,\int_{0}^{2\pi}\, \left(\sum_{m,n=0}^{\infty} (m+1) \left(\sum_{k=0}^{\infty} \mu_{n+k} (k+1) \alpha^{k+m}\right) {r^n e^{in\theta}} {r^m e^{-im\theta}} \right)\ d\theta\,r\, dr \right| \\
		&\quad = 2 (1-\alpha)^{\frac{2}{p'}+\frac{2}{q'}} \int_{0}^{1}  \left(\sum_{n=0}^{\infty}  (n+1) \left(\sum_{k=0}^{\infty} \mu_{n+k}(k+1) \alpha^{k+n}\right) r^{2n+1}\right)\ dr.
	\end{align*}
 Since, for $\alpha \in [0,1) $  and any $n,k$, it holds that 
 $$
 \mu_{n+k} = \int_0^1 t^{n+k} d\mu(t) \geq \int_{\alpha}^1 t^{n+k} d\mu(t) 
 \geq \alpha^{n+k} \mu([\alpha,1)),
 $$ 
 we get that 
	\begin{align*}
		&(1-\alpha)^{\frac{2}{p'}+\frac{2}{q'}-2}\gtrsim\left|\int_{\D} \overline{f_\alpha(w)} \mathcal{H}_\mu(f_\alpha)(w)\ dA(w)\right|\\
		&\quad \geq 2\, (1-\alpha)^{\frac{2}{p'}+\frac{2}{q'}} \, \mu([\alpha,1))\, \int_{0}^{1}  \left(\sum_{n=0}^{\infty}\, (n+1) \left(\sum_{k=0}^{\infty} (k+1) \alpha^{2(k+n)}\right) r^{2n+1}\right) \ dr \\
		&\quad = 2\, (1-\alpha)^{\frac{2}{p'}+\frac{2}{q'}}\, \mu([\alpha,1))\,\left(\sum_{k=0}^{\infty} (k+1) \alpha^{2k}\right)\,\int_{0}^{1}  \left(\sum_{n=0}^{\infty} (n+1) \alpha^{2n} r^{2n+1}\right) \ dr\\
		&\quad = (1-\alpha)^{\frac{2}{p'}+\frac{2}{q'}}\, \mu([\alpha,1))\,\left(1-\alpha^2\right)^{-2}\int_{0}^{1}  \frac{r }{(1-\alpha^2 r^2)^2}\ dr\\
		&\quad \geq  (1-\alpha)^{\frac{2}{p'}+\frac{2}{q'}}\, \mu([\alpha,1))\,\left(1-\alpha^2\right)^{-2}\int_{\alpha}^{1}  \frac{\alpha }{(1-\alpha^2 r^2)^2}\ dr\\
		&\quad \gtrsim (1-\alpha)^{\frac{2}{p'}+\frac{2}{q'}}\mu([\alpha,1))\,\left(1-\alpha \right)^{-2} \frac{\alpha}{(1-\alpha)}\\
		&\quad =\alpha (1-\alpha)^{\frac{2}{p'}+\frac{2}{q'}} \frac{\mu([\alpha,1))}{(1-\alpha)^3}.
			\end{align*}

	Therefore, we have that
		$$
			\frac{\mu([\alpha,1))}{(1-\alpha)} \lesssim 1
		$$
	for all $\alpha\in [0,1)$, that is, $\mu$ is a $1$-Carleson measure.
\end{Prf}	
\newline
\\
\begin{Prf}\textit{the necessary condition for the compactness of the generalized Hilbert operator $\mathcal H_{\mu}$.}
We assume that $\mathcal{H}_{\mu}: RM(p,q)\rightarrow RM(p,q)$ is compact. Let us consider $\{\alpha_n\} \subset [0,1)$ such that $ a_n \to 1, n \to \infty,$ and the family of functions
 $$
 f_{n}(z)=\frac{(1-\alpha_n)^{\frac{1}{p'}+\frac{1}{q'}}}{(1-\alpha_n z)^2}\,, \quad z \in \D\,.
 $$
  Notice that $\sup_n\|f_n\|_{AT^q_p}\lesssim 1$, since $2>\frac{1}{p'}+\frac{1}{q'}>\frac{1}{p}+\frac{1}{q}$. Moreover, it is clear that $f_{n}\rightarrow 0, n \to \infty,$ uniformly on compact sets of the unit disc.  Then, \cite[Lemma 3.7]{Tjani} yields $\lim\limits_{n\rightarrow \infty}\|\mathcal{H}_{\mu}(f_{n})\|_{AT_p^q}=0$. Thus, if $\varepsilon >0$ then there is $n_0 \in\N$ such that for every $n \geq n_0$ , $\|\mathcal{H}_{\mu}(f_n)\|_{AT_p^q}\leq \varepsilon$. Let $ n\geq n_0$. Then, repeating the argument applied in order to prove the  necessary condition of 
  Theorem~\ref{about Hm} (i), we get 
  \begin{align*}
  	\frac{\mu([{\alpha_n},1))}{(1-{\alpha_n})}\lesssim \|\mathcal{H}_{\mu}(f_n)\|_{AT_p^q} \lesssim \varepsilon
  \end{align*}
 that is, $\mu$ is a $1$-Vanishing Carleson measure.
\end{Prf}
\section{Norm estimation of Hilbert operator on $AT_{p}^{q}$.}
In the current section, we will provide estimates of the norm of the Hilbert operator acting on the $AT_{p}^{q}$ spaces. In order to obtain the the upper bound, we will make use of some tools such as the factorization of the analytic tent spaces $AT_p^q$, which can found in \cite{CohnVerbitsky}, and the estimation of the norm of the composition operator for a certain class of tent spaces, which will be proved below. Regarding the lower bound, we will follow the same reasoning as \cite{DJV} to obtain the one that we believe is most appropriate. 

\begin{lemma}
\label{composition on triebel spaces}
    Let $2 \leq p < \infty$ and $\varphi$ an analytic self-map
    of the unit disc $\D$. The composition operator $C_{\varphi}$ is bounded on $AT_p^{\infty}$ and $$\|C_{\varphi}\|\lesssim\frac{1}{(1-|\varphi(0)|)^{\frac{1}{p}}}.$$
\end{lemma}
\begin{proof}
We begin by noting a chain of equivalent norms. The first equivalence follows from \eqref{TentCMDer}, while the second one relies on a description of classical Carleson measures (see \cite[Lemma 3.3]{Garnett}).  
\begin{align*}
    \|f\|_{AT_{p}^{\infty}}^p
    & \asymp
    |f(0)|^p+\sup_I \frac{1}{|I|}\int_{S(I)}|f'(z)|^p (1-|z|)^p dA(z) 
    \\ & \asymp|f(0)|^p+\sup_{\a \in\D} \int_\D  |f'(z)|^p (1-|z|)^p |\phi_\a'(z)|dA(z), 
\end{align*}
where $\phi_\a$ denotes the involutive automorphism $\phi_\a(z)=\frac{\a-z}{1-\overline{\a}z}$. 
Equivalently,
\begin{equation}
\label{triebel norm}
     \|f\|_{AT_{p}^{\infty}}^p \asymp |f(0)|^p+\sup_{\a \in\D} \int_\D  |f'(z)|^p (1-|z|)^{p-1} \log \frac{1}{|\phi_\a (z)|} dA(z).
\end{equation}
By using the last equivalence,
$$\| C_{\varphi}(f)\|_{AT_{p}^{\infty}}^p \asymp |f(\varphi(0))|^p+\sup_{\a \in\D} \int_\D |(f\circ \varphi)'(z)|^p (1-|z|)^{p-1} \log \frac{1}{|\phi_\a (z)|} dA(z),$$
and we will estimate both terms separately. 

First, observe that if $f\in AT_{p}^{\infty} $, then $|f(\varphi(0))|^p\lesssim\frac{\|f\|_{AT_p^{\infty}}^p}{1-|\varphi(0)|}$. 

On the other hand, Schwarz-Pick Lemma (\cite[Lemma 1.2]{Garnett}) and \cite[Corollary 1.3]{Garnett} yield $(1-|z|^2)|\varphi'(z)|\leq (1- |\varphi(z)|^2) $ and $(1-|z|^2)\leq \frac{1+|\varphi(0)|}{1-|\varphi(0)|}  (1- |\varphi(z)|) $, respectively. So, for fixed $\alpha\in \D$, 
\begin{align*}
 & \int_\D |(f\circ \varphi)'(z)|^p (1-|z|)^{p-1} \log \frac{1}{|\phi_\a (z)|} dA(z)
\\& \leq  \int_\D|f'( \varphi(z))|^p (1-|\varphi(z)|)^{p-2} (1-|z|)\log \frac{1}{|\phi_\a (z)|} |\varphi'(z)|^2dA(z)
\\& \leq \frac{1+|\varphi(0)|}{1-|\varphi(0)|} \int_\D|f'( \varphi(z))|^p (1-|\varphi(z)|)^{p-1} \log \frac{1}{|\phi_\a (z)|} |\varphi'(z)|^2\ dA(z).
\end{align*}
Then, by using the non-univalent change of variable formula, due to Shapiro \cite[p. 398]{Sh} and \cite{Smith}, we obtain 
\begin{align*}
    &\int_\D |(f\circ \varphi)'(z)|^p (1-|z|)^{p-1} \log \frac{1}{|\phi_\a (z)|} dA(z) \\& \lesssim \frac{1+|\varphi(0)|}{1-|\varphi(0)|} \int_\D |f'( u)|^p (1-|u|)^{p-1} \log \frac{1}{|\phi_{\varphi(\a)} (u)|} dA(u) 
    \\& \leq \frac{1+|\varphi(0)|}{1-|\varphi(0)|} \sup\limits_{\b \in \D}\int_\D |f'( u)|^p (1-|u|)^{p-1} \log \frac{1}{|\phi_\b (u)|} dA(u)
    \lesssim \frac{1+|\varphi(0)|}{1-|\varphi(0)|} \|f\|_{AT_{ p}^{\infty}}^p,
\end{align*}
where the last equality is followed by \eqref{triebel norm} and the proof is complete.
\end{proof}
\begin{theorem}
   Let $1<p,q<\infty$ such that $p\neq q$, $\frac{1}{p}+\frac{1}{q}<1$ and $p >  2$. Then the Hilbert operator $\mathcal{H}: RM(p,q)\rightarrow RM(p,q)$ is bounded and 
   $$\int_0^1 \frac{dt}{t^{1-\left(\frac{1}{p}+\frac{1}{q}\right)}(1-t)^{\frac{1}{p}+\frac{1}{q}}}\lesssim \|\mathcal{H}\|\lesssim \int_0^1 \frac{dt}{t^{1-\frac{1}{q}}(1-t)^{\frac{1}{p}+\frac{1}{q}}} .$$
\end{theorem}
\begin{proof}

 First, note that $\mathcal{H}$ is well defined and $\mathcal{H}(f)=\mathcal{I}(f)$ on $ AT_p^q$ by  Proposition~\ref{Hmu=Imu}. The integral representation $\mathcal{I}(f)$ can be interpreted as an improper line integral through the well-known change of variables introduced by Diamantopoulos and Siskakis (see \cite{DiS}). This allows us to view $\mathcal{H}$ as an average of weighted composition operators.
 $$\mathcal{H}(f)(z)=\int_0^1\frac{1}{(s-1)z+1} f \left(\frac{s}{(s-1)z+1} ds\right)= \int_0^1 \om_s(z) (f\circ \varphi_s)(z) ds ,$$
 where $\varphi_s(z)=\frac{s}{(s-1)z+1}$ and $\om_s(z)=\frac{1}{(s-1)z+1}$, $z \in \D$. Therefore, to derive an upper estimate of the norm, we may state that
 \begin{equation*}
     \| \H(f)\|_{AT^q_p} \leq \int_0^1 \|\om_s f\circ \varphi_s \|_{AT^q_p}\ ds.
 \end{equation*}

Now, \cite[Theorem 4.1]{CohnVerbitsky} yields that for $f\in AT_p^q$ there exists an outer function $F\in H^q$ and a function $\Phi \in AT^{\infty}_{p}$ such that $f=F\cdot \Phi$, and moreover,
\begin{enumerate}
    \item $\| F \|_{H^q}\lesssim \|f\|_{AT^p_q}$,
    \item $ \|\Phi\|_{AT^{\infty}_{p}}\leq 1$.
\end{enumerate}
Now, note $\om_s (f\circ \varphi_s)= \om_s (F\circ \varphi_s) \cdot (\Phi\circ \varphi_s)$ for each $s \in [0,1)$. 

On the one hand, we are proving $\om_t (F\circ \varphi_t) \in H^q$. For fixed $t\in [0,1)$, $|\om_t(z)|\leq \frac{1}{t}$. So, $\om_t\in H^{\infty}$ for each $t\in [0,1)$. In addition, since $\varphi_t\in \H(\D)$, composition operator $C_{\varphi_t}$ is bounded on $H^q$ and $\|C_{\varphi_t}\|\leq \left(\frac{1+|\varphi_t(0)|}{1-|\varphi_t(0)|}\right)^{\frac{1}{q}}$. Then, bearing in mind (1), $F\circ \varphi_t\in H^q$ and $\|F\circ \varphi_t\|_{H^q}\lesssim \frac{2^{\frac{1}{q}}}{(1-t)^{\frac{1}{q}}}\|f\|_{AT_p^q}$. Therefore, $\om_t (F\circ \varphi_t) \in H^q$ and
\begin{equation}
    \label{1}
 \|\om_t (F\circ \varphi_t) \|_{H^q} \lesssim \frac{1}{t^{1-\frac{1}{q}}}  \frac{1}{(1-t)^{\frac{1}{q}}}\|F\|_{H^q} \lesssim \frac{1}{t^{1-\frac{1}{q}}}  \frac{1}{(1-t)^{\frac{1}{q}}}\|f\|_{AT_p^q}.
\end{equation}

 On the other hand, it follows from Lemma \ref{composition on triebel spaces}  and (2) that $\Phi\circ \varphi_t \in AT^{\infty}_{p}$ and 
 \begin{equation}
 \label{2}
 \|\Phi\circ \varphi_t\|_{AT^{\infty}_{p}}\lesssim \frac{1}{(1-t)^{\frac{1}{p}}}    .
 \end{equation}
 As a consequence of the second part of \cite[Theorem 4.1]{CohnVerbitsky}, \eqref{1} and \eqref{2}, $\om_t (f\circ \varphi_t ) \in AT_p^q$ and $ \|\om_t (f\circ \varphi_t) \|_{H^q}\lesssim \frac{1}{t^{1-\frac{1}{q}}}  \frac{1}{(1-t)^{\frac{1}{p}+\frac{1}{q}}}\|f\|_{AT_p^q}$ for each $t \in [0,1)$. So that, we have just shown the following upper estimate of the classical Hilbert operator norm on tent spaces $AT_p^q$
 $$\|\H\|\lesssim \int_0^1\frac{1}{t^{1-\frac{1}{q}}}  \frac{1}{(1-t)^{\frac{1}{p}+\frac{1}{q}}}dt.$$
 
With the aim of obtaining the lower estimate of the norm, we follow the proof of \cite[Theorem 4]{DJV}. Consider the family of test functions $f_{\alpha}(z)=(1-z)^{-\a}$, $\a<\frac{1}{p}+\frac{1}{q}<1$. Then $f_{\a}\in AT_p^q$ (\cite[Example 2.3]{AgConRodr}). 

As the authors show in  \cite[Theorem 4, p. 2811]{DJV}, essentially a computation involving the change of variables  $\om=\frac{1-rz}{1-z}$ shows that $\H(f_{\a})=f_{\a}\Phi_{\a}$, where $$\Phi_\a(z)=\int_1^{\infty}\frac{d\om}{\om(\om-z)^{1-\a}}.$$
After gaining sufficient insight into this integral, we can assume $\om \in \mathbb{R}$, $\om \geq 1$ and it is obvious that $\Phi_\a$ is an analytic function on  $\overline{\D}\setminus\{1\}$, which attains its maximum modulus at $z=1$.

Let us show that $$\lim\limits_{\a \to \frac{1}{p}+\frac{1}{q} }\|f_\a \|_{AT_p^q}=\infty.$$
Indeed, for fixed $M>0$, let $\delta=\frac{1}{M}$. Consider $0<\frac{1}{p}+\frac{1}{q}-\alpha<\delta$. Then,
\begin{align*}
    \|f_\a \|_{AT_p^q}^q &\gtrsim \int_0^1 \left( \int_{1-\theta}^1\frac{1}{ |1-re^{i\theta}|^{\a p}} dr \right)^{ \frac{q}{p}} \frac{d\theta}{2\pi}=\int_0^1 \left( \int_{1-\theta}^1 \frac{1}{((1-r)^2+ 2r(1-\cos\theta))^{\frac{\a p}{2}}}dr \right)^{\frac{q}{p}} \frac{d\theta}{2\pi}
    \\ & \geq \int_0^1 \left( \int_{1-\theta}^1 \frac{1}{\theta^{\a p}} dr \right)^{\frac{q}{p}}\frac{d\theta}{2\pi} \asymp \int_0^1 \frac{1}{\theta^{\a q -\frac{q}{p}}} d \theta =\frac{1}{ q \left( \frac{1}{p}+\frac{1}{q}-\alpha \right) } > M . 
\end{align*}
Consider $g_{\a}=\frac{f_\a}{\| f_\a \|_{AT_p^q}}$. Now we can assert that this family has the following three properties:\begin{itemize}
    \item[a)] $|g_\a(z)|\geq 0$,
    \item[b)] $\| g_\a\|_{AT_p^q}=1$,
    \item[c)]$|g_\a(z)|\to 0$ uniformly on any compact subset of $\overline{\D}\smallsetminus\{1\}$, as $\a \to \frac{1}{p}+\frac{1}{q}$.
\end{itemize}

Now, we can estimate the following:
\begin{align*}
\|\H(g_\a) \|_{AT_p^q}-|\Phi_{\frac{1}{p}+\frac{1}{q}}(1)|=\|\H(g_\a)\|_{AT_p^q}-\||\Phi_{\frac{1}{p}+\frac{1}{q}}(1)| g_\a\|_{AT_p^q}=\||g_\a|(|\Phi_\a|- |\Phi_{\frac{1}{p}+\frac{1}{q}}(1)|)\|_{AT_p^q}.
\end{align*}
Let us show that $\||g_\a|(|\Phi_\a|- |\Phi_{\frac{1}{p}+\frac{1}{q}}(1)|)\|_{AT_p^q} \to 0$ as $\a \to \frac{1}{p}+\frac{1}{q}$. 
On the one hand, since the function $\Phi_\a$ is continuous on the compact set $\{(z, \a)\in \overline{\D}\times [0,\frac{1}{p}+\frac{1}{q}]\} $, then it is uniformly continuous on this compact. So, for fixed $\varepsilon>0$, there exists $\delta>0$ such that if $z \in D(1, \delta)$ and $\alpha\in \left(\frac{1}{p}+\frac{1}{q}-\delta, \frac{1}{p}+\frac{1}{q}\right)$, then $|\Phi_\a(z)- \Phi_{\frac{1}{p}+\frac{1}{q}}(1)|<\varepsilon$. 
Therefore, 
$$\||g_\a(z)|(|\Phi_\a(z)|- |\Phi_{\frac{1}{p}+\frac{1}{q}}(1)|)\chi_{D(1,\delta)}(z)\|_{AT_p^q}\leq \varepsilon\||g_\a(z)|\chi_{D(1,\delta)}(z)\|_{AT_p^q}\leq \varepsilon.$$

On the other hand, bearing in mind c), there exists $\delta_1
>0$ such that $|g_\a(z)|<\varepsilon$ for all $\a \in \left(\frac{1}{p}+\frac{1}{q}-\delta_1, \frac{1}{p}+\frac{1}{q} \right)$ and $z \in \D\smallsetminus D(1, \delta)$. So that,
\begin{align*}
\||g_\a(z)|(|\Phi_\a(z)|- |\Phi_{\frac{1}{p}+\frac{1}{q}}(1)|)\chi_{\D \smallsetminus D(1,\delta)}(z)\|_{AT_p^q}&\leq \varepsilon\|(|\Phi_\a(z)|- |\Phi_{\frac{1}{p}+\frac{1}{q}}(1)|)\chi_{\D \smallsetminus D(1,\delta)}(z)\|_{AT_p^q} 
\\ & \leq  \varepsilon (|\Phi_\a(1)|- |\Phi_{\frac{1}{p}+\frac{1}{q}}(1)|)
\\ & <\varepsilon \left(\frac{\pi}{\sin \left( \pi \left( \frac{1}{p}+\frac{1}{q} -\delta_1 \right) \right)} +\frac{\pi}{\sin \left(\pi \left( \frac{1}{p}+\frac{1}{q} \right) \right)}\right) .
\end{align*}  
Finally, a simple computation shows that 
$$\Phi_{\frac{1}{p}+\frac{1}{q}}(1)=\int_1^{\infty}\frac{ds}{s(s-1)^{1-\frac{1}{p}-\frac{1}{q}}}= \int_0^1 s^{-\left(\frac{1}{p}+\frac{1}{q}\right)} (1-s)^{\frac{1}{p}+\frac{1}{q}-1}ds=\int_0^1 s^{\frac{1}{p}+\frac{1}{q}-1} (1-s)^{-\left(\frac{1}{p}+\frac{1}{q}\right)}ds.$$
All in all, we have just proved the lower estimate of the norm
$$\int_0^1 \frac{dt}{t^{1-\left(\frac{1}{p}+\frac{1}{q}\right)}(1-t)^{\frac{1}{p}+\frac{1}{q}}}\lesssim \|\mathcal{H}\|.$$
\end{proof}

\begin{remark}
Note that the proof of Theorem \ref{norm estimation} is based on the integral representation of the Hilbert operator, $\H(f)=\I(f)$, $f\in AT_p^q$, which we have shown to be valid if $\frac{1}{p}+\frac{1}{q}<1$ and $p > 2$. However, the proof also works for $p=2$ and this yields an estimation of the norm of the integral operator $\I$ for $\frac{1}{p}+\frac{1}{q}<1$ and $p = 2$.
\end{remark}

\end{document}